\documentclass[reqno,12pt]{amsart}

\usepackage[pagebackref]{hyperref}
\usepackage[left=2.5cm,right=2.5cm]{geometry}
\usepackage{stackrel}
\newcommand{\Ent}{{\rm Ent}}

\def\Var{{\rm Var}}
\def\E{\mathbb E}
\def\p{\mathbb P}

\newcommand*{\ind}[1]{\mathbf{1}_{\{#1\}}}

\newcommand{\D}{\mathbf{D}}
\newcommand{\R}{\mathbb{R}}

\newtheorem{lemma}{Lemma}[section]

\newtheorem{theorem}[lemma]{Theorem}
\newtheorem{prop}[lemma]{Proposition}
\newtheorem{cor}[lemma]{Corollary}
\newtheorem{defi}[lemma]{Definition}

\newcommand{\bn}{\Big\|}

\title[Moment estimates implied by modified log-Sobolev inequalities]{Moment estimates implied by modified log-Sobolev inequalities.}

\thanks{Research partially supported by Polish Ministry of Science and Higher Education Iuventus Plus
  Grant no. IP 2011 000171.}

\author{Rados{\l}aw Adamczak}%
\address[RA]{Institute of Mathematics, University of Warsaw, ul. Banacha 2,
  02-097 Warszawa, POLAND. R.Adamczak@mimuw.edu.pl}

\author{Witold Bednorz}%
\address[WB]{Institute of Mathematics, University of Warsaw, ul. Banacha 2,
  02-097 Warszawa, POLAND. W.Bednorz@mimuw.edu.pl}

\author{Pawe{\l} Wolff} %
\address[PW]{Institute of Mathematics, University
  of Warsaw, ul. Banacha 2, 02-097 Warszawa, POLAND; and Institute of Mathematics, Polish Academy of Sciences, ul. \'Sniadeckich 8, 00-956 Warszawa, POLAND. P.Wolff@mimuw.edu.pl}

\begin{document}

\begin{abstract}

We study a class of logarithmic Sobolev inequalities with a general form of the energy functional. The class generalizes various examples of modified logarithmic Sobolev inequalities considered previously in the literature. Refining a method of Aida and Stroock for the classical logarithmic Sobolev inequality, we prove that if a measure on $\R^n$ satisfies a modified logarithmic Sobolev inequality then it satisfies a family of $L^p$-Sobolev-type inequalities with non-Euclidean norms of gradients (and dimension-independent constants). The latter are shown to yield various concentration-type estimates for deviations of smooth (not necessarily Lipschitz) functions and measures of enlargements of sets corresponding to non-Euclidean norms. We also prove a two-level concentration result for functions of bounded Hessian and measures satisfying the classical logarithmic Sobolev inequality.

\end{abstract}

\subjclass{60E15, 26D10}
\keywords{Concentration of measure, modified logarithmic Sobolev inequalities}

\maketitle

\section{Introduction}

Concentration of measure inequalities constitute one of the strongest and most widely used tools in modern high dimensional probability, geometry and analysis, crucial in establishing e.g. limit theorems or existence proofs by probabilistic method. Their importance was first noted in the 1970s and since then many powerful approaches have been established, which allow to prove concentration results, such as isoperimetric, transportation or functional inequalities (we refer to the monograph \cite{LedouxConcBook} by Ledoux for an overview). Among the functional inequalities approaches, the two which have proven particularly useful are those based on the Poincar\'e and logarithmic Sobolev inequalities. Recall that a Borel probability measure $\mu$ on $\R^n$ satisfies the Poincar\'e inequality with constant $D$ if
\begin{displaymath}
\Var_\mu f \le D \E_\mu |\nabla f|^2
\end{displaymath}
for all sufficiently smooth functions $f \colon \R^n \to \R$, whereas the logarithmic Sobolev inequality holds if
for all such $f$,
\begin{displaymath}
\Ent_\mu f^2 \le D \E_\mu |\nabla f|^2,
\end{displaymath}
where $\Ent_\mu f^2 = \E_\mu f^2\log f^2 - \E_\mu f^2\log \E_\mu f^2$ is the usual entropy of $f^2$ (throughout the article we use the probabilistic notation, treating $f$ as a random variable on the probability space $(\R^n, \mathcal{B}(\R^n),\mu)$, in particular $\E_\mu$ denotes integration with respect to $\mu$). Above and in the rest of the paper $|\cdot|=|\cdot|_2$ always denotes the standard Euclidean norm in $\R^n$.

As is well known, the Poincar\'e inequality yields subexponential concentration of Lipschitz functions, whereas the logarithmic Sobolev inequality implies sub-Gaussian estimates. There are also other functional inequalities, based either on a modification of the variance functional in the Poincar\'e inequality \cite{MR954373, MR1796718} or a modification of the right-hand side in the logarithmic Sobolev inequalities~\cite{MR1800062,MR2198019,MR2351133}, which yield concentration estimates with super-Gaussian rates or with rates between subexponential and sub-Gaussian. The general form of such log-Sobolev inequalities is
\begin{equation}\label{eq:mLSI-intro}
\Ent_\mu f^2 \le D \E_\mu \Psi\Big(\frac{\nabla f }{f}\Big)f^2
\end{equation}
for an appropriate function $\Psi \colon \R_+ \to \R_+$ (we postpone the introduction of technical conditions on the function $\Psi$ to subsequent sections).

Together with the discovery of two-level concentration inequalities by Talagrand (initially for the exponential distribution~\cite{MR1122615}), which improve the estimates based on Poincar\'e inequality and provide sub-Gaussian estimates for relatively small deviations and subexponential bounds for larger ones, a natural question arose whether results of this type could also be obtained via functional inequalities. It was soon answered in the affirmative by Bobkov and Ledoux \cite{BobLed_exp} who derived new modified logarithmic Sobolev inequalities, which were subsequently extended by Gentil, Guillin and Miclo \cite{MR2198019, MR2351133} to inequalities yielding other (two-level and also more general) types of concentration.

Concentration estimates are usually derived from modified logarithmic Sobolev inequalities via differential inequalities on the log-Laplace transform of the function, a method commonly known as the Herbst argument. This method, being very elegant and powerful is however restricted to functions with finite Laplace transform, such as Lipschitz functions. For functions which do not satisfy the Lipschitz condition one can still use a modification of the Herbst approach, proposed by Aida and Stroock \cite{MR1258492}. It relies on the analysis of moments and provides $L_p$-type Sobolev inequalities, which can yield concentration, provided that one controls the gradient of a function. More precisely, Aida and Stroock proved that the logarithmic Sobolev inequality implies inequalities of the form
\begin{displaymath}
\|f - \E_\mu f\|_p \le C_D\sqrt{p}\bn |\nabla f| \bn_p
\end{displaymath}
for $p \ge 2$, where $\|g\|_p = (\E_\mu |g|^p)^{1/p}$ is the $p$-th moment of the function $g$. Clearly, controlling all moments of $\nabla f$ allows then to derive concentration results.

The aim of this work is to provide a uniform framework which would allow to treat the aforementioned modified logarithmic inequalities and provide concentration estimates with general profiles of the deviation bound for functions which are not necessarily Lipschitz. Our approach is based on a further refinement of the method by Aida and Stroock, which gives $L_p$-type Sobolev inequalities with non-Euclidean, $p$-dependent norms of the gradient on the right-hand side. The form of these norms corresponds to the type of concentration satisfied by the measure. As a particular case we obtain moment inequalities for smooth functions, which generalize moment estimates for linear combinations of independent random variables, derived by Gluskin and Kwapie\'n \cite{GK} (see Theorem \ref{thm:GK} below).

We remark that even in the classical Euclidean framework under defective log-Sobolev inequalities, our results improve certain aspects of the work by Aida and Stroock.

The precise form of the inequalities we obtain depends on the function $\Psi$ in the modified logarithmic Sobolev inequality~\eqref{eq:mLSI-intro} and to describe it we need to introduce some technical notation. For this reason we postpone the precise formulation to subsequent sections and now we just announce that we will obtain inequalities of the form
\begin{displaymath}
\|f - \E_\mu f\|_p \le C_D\bn|\nabla f|_{\Psi_p}\bn_p,
\end{displaymath}
for some norm $|\cdot|_{\Psi_p}$ on $\R^n$ associated with $p$ and the function $\Psi$. The geometry of the norms $|\cdot|_{\Psi_p}$ will be responsible for the character of concentration of measure valid for the measure $\mu$. In the classical case, considered by Aida and Stroock, we have simply $|x|_{\Psi_p} = \sqrt{p}|x|$.

As corollaries we obtain concentration results for not necessarily Lipschitz functions as well as bounds on the size of enlargement of sets in a setting more general than considered before. We also prove some concentration results for Banach space valued polynomial chaos in the case of not necessarily product measures, extending previous work by Borell \cite{Bo}, Arcones-Gin\'e \cite{MR1201060}, {\L}ochowski \cite{Loch} and Adamczak \cite{AdLogSobConv}. Additionally we derive comparison principles for real-valued polynomials (or more generally functions with bounded derivatives of higher order) generalizing previous estimates by Adamczak and Wolff \cite{AdLogSobConv} and a two-level concentration result for functions with bounded Hessian and measures satisfying the classical logarithmic Sobolev inequality.

\medskip
\medskip
\paragraph{\bf Organization of the paper} In Section \ref{sec:preliminaries} we introduce the general framework for the inequalities we consider. Section \ref{sec:main_results} is devoted to the presentation of our results. The proofs are deferred to Section \ref{sec:proofs}.

\section{Preliminaries \label{sec:preliminaries}}
\subsection{Basic notation}
We will be working mostly with a fixed probability measure $\mu$, therefore we will denote $\|f\|_p = (\E_\mu |f|^p)^{1/p}$, suppressing the dependence on $\mu$ in the notation.

Unless otherwise stated, $x_i$, $i=1,\ldots,n$, will denote coordinates of a point $x\in \R^n$, i.e. $x = (x_1,\ldots,x_n)$. To distinguish norms on $\R^n$ from the notation for moments, we will denote the former with single bars, e.g. for $r \ge 1$, $|\cdot|_r$ will stand for the $\ell_r^n$ norm, defined by $|x|_r = (\sum_{i=1}^n |x_i|^r)^{1/r}$. Other important norms will be introduced in the sequel. In the case of $r=2$ we will often suppress the subscript $r$ and write simply $|\cdot|$ for $|\cdot|_2$.

By $C,c$ we will denote universal constants, whereas the notation $C(a)$ or $C_a$ will be used for constants depending only on a parameter $a$. The values of constants may differ between occurrences.

\subsection{Generalized Orlicz functions and modified logarithmic Sobolev inequalities}
To formulate our results let us first introduce the general abstract form of the inequalities we will consider. Next we will illustrate it with examples of inequalities known in the literature, which fit our framework.

In what follows we will consider generalized Orlicz functions on $\R^n$, satisfying some standard technical conditions given in the following

\begin{defi}
We will say that a function $\Psi\colon \R^n \to \R_+ \cup \{\infty\}$ satisfies the condition (C) if the following holds
\begin{enumerate}
\item[(C1)\label{C:zero}] $\Psi(0) = 0$ and $\Psi$ is continuous at $0$,
\item[(C2)\label{C:positive}] $\Psi(x) > 0$ for $x \neq 0$,
\item[(C3)\label{C:infty}] $\lim_{|x|\to \infty} \Psi(x) = \infty$,
\item[(C4)\label{C:left-cont}] for every $x \in \R^n$, the function $t\mapsto \Psi(tx)$ is left-continuous on $(0,\infty)$,
\item[(C5)\label{C:non-dec}] for every $x \in \R^n$, the function $t \mapsto \Psi(tx)/t$ is non-decreasing on $(0,\infty)$,
\item[(C6)\label{C:symmetric}] $\Psi$ is symmetric, i.e. $\Psi(x) = \Psi(-x)$ for all $x \in \R^n$.
\end{enumerate}
\end{defi}

Consider a probability measure $\mu$ on $\R^n$, absolutely continuous with respect to the Lebesgue measure. The general class of functional inequalities we will consider is described in the following definition.

\begin{defi}
Given a function $\Psi$, satisfying the condition (C) and a positive constant $D$ we will say that $\mu$ satisfies the \emph{modified logarithmic Sobolev inequality} $mLSI(\Psi,D)$ if for every bounded locally Lipschitz function $f \colon \R^n \to (0,\infty)$,
\begin{align}\label{eq:mLSI}
\Ent_\mu f^2 \le D \E_\mu \Psi\Big(\frac{\nabla f }{f}\Big)f^2.
\end{align}
\end{defi}
Note that by the Rademacher theorem, $\nabla f$ exists $\mu$-a.s. By standard arguments one can show that $\mu$ satisfies $mLSI(\Psi,D)$ if and only if the above inequality is satisfied by all smooth functions of bounded support.

We will also consider defective versions of the logarithmic Sobolev inequalities.

\begin{defi}
Given a function $\Psi$, satisfying the condition (C) and constants $D, d \ge 0$ we will say that $\mu$ satisfies the \emph{defective modified logarithmic Sobolev inequality} $dmLSI(\Psi,D,d)$, if for every bounded locally Lipschitz function $f \colon \R^n \to (0,\infty)$,
\begin{align}\label{eq:dmLSI}
\Ent_\mu f^2 \le D \E_\mu \Psi\Big(\frac{\nabla f}{f}\Big)f^2 + d\E_\mu f^2.
\end{align}
\end{defi}

As already mentioned, inequalities of this form have been considered by many authors starting from the classical work by Stam \cite{MR0109101}, Federbush \cite{Federbush}, Gross \cite{MR0420249} on the logarithmic Sobolev inequality, which in our language corresponds to the choice $\Psi(x) = |x|^2$, where $|\cdot|$ is the standard Euclidean norm. Modified versions were first considered by Bobkov-Ledoux \cite{MR1800062}, then e.g. by Bobkov-Zegarli\'nski \cite{MR2146071}, Gentil-Guillin-Miclo \cite{MR2198019,MR2351133}, Barthe-Roberto \cite{MR2430612} and Barthe-Kolesnikov \cite{MR2438906}. The defective versions were investigated e.g. by Rothaus \cite{MR812396}, Bobkov-Zegarlinski \cite{MR2146071}, Barthe-Kolesnikov \cite{MR2438906}, who obtained general criteria under which one can infer the non-defective version from the defective one. For instance it is known that under some additional conditions on the function $\Psi$, the defective inequality implies the non-defective version if one assumes certain Poincar\'e type inequalities.

Below we present the best known examples of modified log-Sobolev inequalities.
\begin{itemize}
\item The inequality \eqref{eq:mLSI} with $\Psi(x) = \|x\|^q$, where $q \in (1,2]$ and $\|\cdot\|$ is some norm on $\R^n$ was introduced by Bobkov and Ledoux in \cite{MR1800062}.
\item In \cite{BobLed_exp} Bobkov and Ledoux considered $\Psi(x) = \sum_{i=1}^n H(x_i)$, where
\begin{displaymath}
H(x) = \left\{\begin{array}{ccc}
x^2 &\textrm{for}& |x| \le 1/2\\ \infty&\textrm{for}& |x| > 1/2.
\end{array}\right.
\end{displaymath}
This inequality was used to recover Talagrand's concentration inequality for the exponential distribution \cite{MR1122615}. In \cite{MR2198019,MR2351133} Gentil, Gullin and Miclo generalized the Bobkov-Ledoux inequality, by considering $H(x) = x^2\ind{|x|\le 1} + \Phi(|x|)\ind{|x| > 1}$ for a convex function $\Phi$. When $\Phi(x)/x^2$ is non-decreasing on the positive half-line, the characterization of measures on $\R$ which satisfy $mLSI(H,D)$ for some finite $D$ was obtained in \cite{MR2430612} (the characterization of the classical case $\Psi(x) = x^2$ was obtained earlier in the seminal paper \cite{MR1682772} by Bobkov-G{\"o}tze).
\item The inequalities \eqref{eq:mLSI} and \eqref{eq:dmLSI} for general $\Psi$ corresponding to a measure $\mu(dx) =  e^{-V(x)}dx$ on $\R^n$, under certain Bakry-Emery type conditions relating $\Psi$ and $V$ were studied e.g. by Barthe-Kolesnikov \cite{MR2438906}, Gentil \cite{MR2487856}, Shao \cite{MR2520723}.
\end{itemize}

Since the aforementioned articles introduce many different approaches for proving modified logarithmic Sobolev inequalities and the presentation of all of them is beyond the scope of this paper let us only mention that there is a multitude of examples of measures satisfying the inequalities in question. Currently available tools allow to both find mild sufficient conditions for a measure to satisfy the inequality \eqref{eq:mLSI} (resp. \eqref{eq:dmLSI}) with a given function $\Psi$ or starting from a measure find an appropriate $\Psi$ so that the inequality \eqref{eq:mLSI} (resp. \eqref{eq:dmLSI}) holds. Moreover, the usual tensorization and perturbation arguments developed for the classical logarithmic Sobolev inequality \cite{logSob} work also in the modified setting and allow to construct further examples. In particular, the one-dimensional characterizations of modified logarithmic Sobolev inequalities allow to consider inequalities on $\R^n$ with $\Psi(x) = \sum_{i=1}^n H(x_i)$, first for product measures and then for their bounded perturbations.

\subsection{Families of Orlicz norms}

Let us introduce another notion we need to formulate our results, namely a family of (quasi-)norms related to the function $\Psi$. In the Sobolev type inequalities we are about to derive  these norms will be applied to the gradient of a function.

For $p > 0$ define
\begin{displaymath}
\Psi_p(x) = \frac{1}{p}\Psi(px).
\end{displaymath}
Assume $\Psi$ satisfies the condition (C). Then so does $\Psi_p$ for all $p > 0$ and we consider a family of (quasi-)norms $|\cdot|_{\Psi_p}$ on $\R^n$, defined as
\begin{displaymath}
|x|_{\Psi_p} = \inf\{a > 0\colon \Psi_p(x/a) \le 1 \} = \inf\{a > 0\colon \Psi(px/a) \le p\}.
\end{displaymath}
It is easy to see that $|x|_{\Psi_p}$ is indeed a quasi-norm on $\R^n$, i.e. $|x|_{\Psi_p} = 0$ iff $x = 0$, $|tx|_{\Psi_p} = |t||x|_{\Psi_p}$ for $t \in \R$ and $|x+y|_{\Psi_p} \le K_{\Psi_p}(|x|_{\Psi_p} + |y|_{\Psi_p})$ for some constant $K_{\Psi_p}$. If $\Psi$ is in addition convex, then $|\cdot|_{\Psi_p}$ is a norm on $\R^n$. In what follows we will refer to the functional $|\cdot|_{\Psi_p}$ as the $\Psi_p$-norm or simply the norm, even if the function $\Psi_p$ is not necessarily convex.
Note also that from~(C5) it follows that the norms $|\cdot|_{\Psi_p}$ are non-decreasing in $p$, i.e. for $0 < p < q$ and any $x \in \R^n$,
\begin{align}\label{ineq:monotonicity}
  |x|_{\Psi_p} \le |x|_{\Psi_q}.
\end{align}

\paragraph{\bf Examples} To demonstrate the reasons for introducing the above abstract definition of the norms $|\cdot|_{\Psi_p}$ and to show their role in the derivation of concentration of measure results, we will now list some families of norms corresponding to special choices of the function $\Psi$ and present some special cases of known Sobolev type inequalities.

\begin{enumerate}
\item Clearly, if $\Psi$ is homogeneous, in particular if it is a norm on $\R^n$, then for any $p > 0$, $|\cdot|_{\Psi_p} = \Psi$. This shows that to obtain interesting families of norms, $|\cdot|_{\Psi_p}$, which can be used to control the behaviour of moments of random variables, one needs to consider functions $\Psi$ which grow faster than linearly.

\item If $\Psi(x) = \|x\|^\alpha$ for some norm $\|\cdot\|$ on $\R^n$ and $\alpha > 1$, then $|x|_{\Psi_p} = p^{1/\alpha^\ast}\|x\|$, where $\alpha^\ast$ is the H\"older conjugate of $\alpha$. This simple example can already illustrate the role played by the norms $|\cdot|_{\Psi_p}$ in our estimates. For instance, it is well known~\cite{PisierProbabMethods} that for a standard Gaussian measure $\gamma$ on $\R^n$, for every smooth function $f \colon \R^n \to \R$ and every $p \ge 2$,
    \begin{align}\label{eq:Maurey_Pisier}
    \|f - \E_\gamma f\|_p\le C\sqrt{p}\bn |\nabla f|\bn_p
    \end{align}
    for some absolute constant $C$ (where the moments $\|\cdot\|_p$ are calculated with respect to $\gamma$). Note that if $|\nabla f| \le L$ on $\R^n$, then by applying the Chebyshev inequality in $L_p$ and optimizing in $p$, one recovers (up to constants) the classical Gaussian concentration inequality, i.e.
    \begin{displaymath}
    \gamma(|f - \E_\gamma f|\ge t) \le 2\exp(-ct^2/L^2)
    \end{displaymath}
    for some universal constant $c$.
    It is easy to see that for $\Psi(x) = |x|^2$, the right hand side of \eqref{eq:Maurey_Pisier} can be written as $C\Big\||\nabla f|_{\Psi_p}\Big\|_p$, so the Sobolev inequality~\eqref{eq:Maurey_Pisier} is equivalent to
    \begin{align}\label{eq:modSobolev}
   \|f - \E_\gamma f\|_p \le C\Big\||\nabla f|_{\Psi_p}\Big\|_p.
    \end{align}

\item Let us now consider an example corresponding to a two-level concentration of measure. Let $\Psi(x) = \sum_{i=1}^n (|x_i|^2\ind{|x_i|\le 1} + |x_i|^r\ind{|x_i|>1})$ for some $r \in [2,\infty)$. Recall that by $|\cdot|_r$ we denote the $\ell_r^n$ norm, i.e. for $x = (x_1,\ldots,x_n)$, $|x|_r = (\sum_{i=1}^n |x_i|^r)^{1/r}$. One can show that
    \begin{displaymath}
    |x|_{\Psi_p} \simeq \sqrt{p}|x| + p^{1/r^\ast}|x|_r,
    \end{displaymath}
    where $\simeq$ denotes two-sided estimates matching up to a universal multiplicative constant.

    It turns out that the expression given above appears in Sobolev inequalities leading to two-level tail estimates. Namely, in \cite{Nonlipschitz} it is proved that
    if $r \in [2,\infty)$ and a measure $\mu$ on $\R^n$ satisfies the $mLSI(\Psi,D)$, then for all smooth functions $f$ and $p \ge 2$,
    \begin{displaymath}
    \|f - \E_\mu f\|_p\le C_{r,D} \Big(\sqrt{p}\Big\| |\nabla f|_2 \Big\|_p + p^{1/r^\ast} \Big\| |\nabla f|_r\Big\|_p \Big)\simeq C_{r,D} \Big\||\nabla f|_{\Psi_p}\Big\|_p,
    \end{displaymath}
    (where the moments $\|\cdot\|_p$ are calculated with respect to $\mu$), which implies that for any Lipschitz function $f$ and $t > 0$, we have
    \begin{align}\label{eq:two-level_large_r}
    \mu(|f - \E_\mu f| \ge t) \le 2\exp\Big(-c_{r,D} \min\Big(\frac{t^2}{a^2},\frac{t^{r^\ast}}{b^{r^\ast}}\Big)\Big),
    \end{align}
    where $a = \sup_{x \in \R^n}  |\nabla f(x)|_2$, $b = \sup_{x \in \R^n} |\nabla f(x)|_r$. This corresponds to Talagrand's two-level concentration inequality. As we will see, a similar moment bound holds also for $r \in (1,2)$ and even for more general functions $\Psi$. We remark that for $r \in [1,2]$, we have
    \begin{displaymath}
    |x|_{\Psi_p} \simeq p^{1/r^\ast} |(x_i^\ast)_{i=1}^{\lfloor p\rfloor}|_r + \sqrt{p}|(x_i^\ast)_{i=\lfloor p\rfloor +1}^n|_2
    \end{displaymath}
    where $x_1^\ast \ge \ldots \ge x_n^\ast$ is the non-increasing rearrangement of the sequence $|x_1|,\ldots,|x_n|$.

\item Let us remark that in the special case, when $\mu$ is a product of measures on $\R$ with log-concave tails and $f$ is a linear functional, the inequalities of the form
\begin{displaymath}
\|f - \E f\|_p \le \Big\||\nabla f|_{\Psi_p}\Big\|_p
\end{displaymath}
have been proved by Gluskin and Kwapie\'n in \cite{GK}. In Section \ref{sec:real_polynomials} we will use their result (which we recall in Theorem \ref{thm:GK}) to give an interpretation of our results in terms of auxiliary i.i.d. sequences.
\end{enumerate}

\section{Main results\label{sec:main_results}}

In this section we will present all our results, deferring their proofs to Section \ref{sec:proofs}.

\subsection{Standing assumptions} Let us first state the main assumptions we are going to use throughout the article.

All the measures we will consider are assumed to be absolutely continuous with respect to the Lebesgue measure and this assumption will not be explicitly stated in all the theorems.

Our usual assumption on the function $\Psi$, beside the condition (C), will be following growth condition: for some $1<\alpha \le 2\le \beta < \infty$ and $K \ge 1$,
\begin{align}\label{ineq:growth-condition}
  \tag{$G_{K,\alpha, \beta}$}
 \forall_{x \in \R^n \setminus\{0\}} \ \forall_{t \ge 1} \quad K^{-1} t^\alpha \le \frac{\Psi(tx)}{\Psi(x)} \le K t^\beta.
\end{align}
Note that the condition~\eqref{ineq:growth-condition} is stable under taking max or sum of functions $\Psi$, and if $\Psi$ satisfies \eqref{ineq:growth-condition} then so does $\Psi_p$ for any $p > 0$.

\subsection{Sobolev type inequalities}

Let us now present our main results, i.e. Sobolev type inequalities, which constitute a basis for all the subsequent corollaries.

\subsubsection{The defective case \label{sec:defective}}
We will start with moment estimates implied by defective inequalities. Recall the definition of the inequality $dmLSI(\Psi,D,d)$ given in formula \eqref{eq:dmLSI}. The proofs of results of this section are provided in Section \ref{sec:defectiv_proofs}.

\begin{theorem}\label{thm:Sobolev}
Assume that $\Psi \colon \R^n \to \R$ satisfies the condition (C) and \eqref{ineq:growth-condition} for some $1<\alpha \le 2\le \beta < \infty$ and $K \ge 1$.
Let $\mu$ be a probability measure on $\R^n$ satisfying $dmLSI(\Psi,D,d)$. Then for all locally Lipschitz functions $f\colon \R^n \to \R$ and all $p \ge \beta$,
\begin{align}\label{eq:mSobolev}
\|f\|_p \le e^{2d/\beta} \|f\|_\beta + \frac{2e}{\alpha-1}\big((KD)^{1/\alpha} \lor (KD)^{1/\beta}\big) \bn |\nabla f|_{\Psi_p}\bn_p.
\end{align}
\end{theorem}
\phantom{aaaa}
\medskip
Let us remark that for any $p \ge \beta$ and any $q \in (0,\beta)$ one can actually obtain
\begin{align}\label{eq:mSobolev_smaller_moment}
\|f\|_p \le 2^{\frac{(p-q)\beta}{(p-\beta)q}} e^{2d\frac{p-q}{(p-\beta)q}} \|f\|_q +
   \frac{2e}{\alpha-1}\big((KD)^{1/\alpha} \lor (KD)^{1/\beta}\big)\bn |\nabla f|_{\Psi_p}\bn_p.
\end{align}
This inequality is a simple consequence of the well-known Lemma~\ref{le:moment_comparison} stated in Section~\ref{sec:proofs}.
Note that the constant $2^{\frac{(p-q)\beta}{(p-\beta)q}} e^{2d\frac{p-q}{(p-\beta)q}}$ obtained with the Lemma \ref{le:moment_comparison} explodes when $q \to 0$ or $p \to \beta$.
We do not know if under the assumption of the above theorem, one can prove that for all $p \ge 2$,
\begin{displaymath}
\|f\|_p \le C(D,\alpha,\beta)\Big(\|f\|_2 + \Big\| |\nabla f|_{\Psi_p}\Big\|_p\Big).
\end{displaymath}
Fortunately for the concentration of measure inequalities, for fixed $\alpha,\beta$, it is enough to control the growth of $\|f - \E f\|_p$ for $p > \beta$. Such a bound will be obtained in the non-defective case. It would be interesting to know if one can obtain meaningful Sobolev inequalities for the case $\beta = \infty$, which corresponds to the Bobkov-Ledoux inequality (satisfied e.g. for the product exponential distribution). In \cite{Nonlipschitz} it has been conjectured that in this case (for $d = 0$),
\begin{displaymath}
\|f - \E f\|_p \le C_D \Big(\sqrt{p}\bn |\nabla f|_2\bn + p\bn |\nabla f|_\infty\bn_p\Big)
\end{displaymath}
and a weaker inequality was proved, with the second term on the right hand side replaced by $p\bn |\nabla f|_\infty\bn_\infty$.

On the other hand, the reason for excluding the case $\alpha=1$ is made clear by the following proposition and the example below.

\begin{prop}\label{prop:alpha_equal_1}
Assume that $\Psi \colon \R^n \to \R$ satisfies the condition (C) and for some $\beta \ge 2$ and $K \ge 1$,
\begin{displaymath}
\forall_{x \in \R^n \setminus\{0\}} \ \forall_{t \ge 1} \quad \frac{\Psi(tx)}{\Psi(x)} \le K t^\beta.
\end{displaymath}
Let $\mu$ be a probability measure on $\R^n$ satisfying $dmLSI(\Psi,D,d)$. Then for all locally Lipschitz functions $f\colon \R^n \to \R$ and all $p \ge \beta$,
\begin{align}\label{ineq:moments-estimate-alpha-1}
\|f\|_p \le e^{2d/\beta} \|f\|_\beta + 6\log(p)\big(D \lor (KD)^{1/\beta}\big) \bn |\nabla f|_{\Psi_p}\bn_p.
\end{align}
\end{prop}

\paragraph{\bf Example:} Let $\nu$ be a probability measure on $\R$ with the distribution function
\begin{equation}
\label{distribution-function-nu}
  F_\nu(x) = \begin{cases}
    \frac12 e^{-e^{-x}+1}, & \text{for $x < 0$,} \\
    1-\frac12 e^{-e^x+1}, & \text{for $x \ge 0$.}
  \end{cases}
\end{equation}
\begin{prop}\label{prop:mLSI-for-nu}
The measure $\nu$ defined by \eqref{distribution-function-nu} satisfies $mLSI(\Psi, 2)$ with $\Psi(x) = |x|$. Also, for every locally Lipschitz function $f \colon \R \to \R$ and $p > 1$,
\begin{equation}
\label{ineq:moments-estimate-for-nu}
\|f\|_{L^p(\nu)} \le \|f\|_{L^1(\nu)} + \log (p) \|f'\|_{L^p(\nu)}.
\end{equation}
Moreover, if $f(x) = x$, then for $p \ge 1$, $\|f\|_{L^p(\nu)} \ge (2e)^{-1}\log(p)$.
\end{prop}

The `moreover' part of the above proposition shows in particular that for $p\to \infty$, the $\log (p)$ factor in~\eqref{ineq:moments-estimate-alpha-1} or \eqref{ineq:moments-estimate-for-nu} cannot be improved.

\paragraph{\bf Further examples:}
\begin{enumerate}
\item If $\Psi(x) = |x|^2$ we are in the setting of the classical defective logarithmic Sobolev inequality. A result by Aida-Stroock \cite{MR1258492} says that in this case for $p \ge 2$,
\begin{align}\label{eq:Aida_stroock}
\|f\|_p^2 \le e^{2d/p^\ast} \Big(\|f\|_{2}^2 + D(p-2)\bn |\nabla f|\bn_p^2\Big).
\end{align}
On the other hand, Theorem~\ref{thm:Sobolev}, specialized to this case, asserts that for $p \ge 2$,
\begin{align}\label{eq:classical}
\|f\|_p  \le e^{d}\|f\|_2 + 2e\sqrt{D} \sqrt{p} \bn |\nabla f|\bn_p.
\end{align}
Let us note that the constant in front of the term involving $\nabla f$ in our inequality does not depend on $d$, which is not the case in \eqref{eq:Aida_stroock}. Therefore, even though the Aida-Stroock bound may behave in a better way for $p$ close to 2, our estimate (via Chebyshev's inequality optimized over $p$) shows that the large deviation behaviour of functions with polynomial growth of moments of $\nabla f$ can be controlled independently of $d$, which does not seem to follow from \eqref{eq:Aida_stroock}. For instance, if $\|f\|_2 < \infty$ and $\bn |\nabla f|\bn_p \le A p^\gamma$ for some $\gamma \ge 0$, then we obtain
\begin{displaymath}
\limsup_{t \to \infty} \frac{1}{t^{2/(1+2\gamma)}}\log \p(|f| \ge t) \le - c_{A,D}
\end{displaymath}
for some (explicit) constant $c_{A,D}>0$.

We remark that an improvement of the Aida-Stroock result for Lipschitz functions was obtained by Rothaus in \cite{MR1452824}.

\item Consider now $\Psi(x) = \sum_{i=1}^n (|x_i|^2\ind{|x_i|\le 1} + |x_i|^r \ind{|x_i|>1})$ for some $r \in (1,\infty)$, which corresponds to the modified logarithmic Sobolev inequality introduced in \cite{MR2198019} and \cite{MR2438906} for $r \ge 2$ and in \cite{MR2487856} for $r < 2$. In the former case, the inequalities of Theorem \ref{thm:Sobolev} read as
    \begin{displaymath}
    \|f\|_p \le e^{2d/r}\|f\|_r + 2e (D^{1/2}\vee D^{1/r}) \Big(\sqrt{p}\bn |\nabla f|_2\bn_p + p^{1/r^\ast}\bn |\nabla f|_r\bn_p\Big)
    \end{displaymath}
    for $p \ge r$ (the assumption ($G_{K,\alpha,\beta}$) is satisfied with $\alpha = 2$, $\beta = r$ and $K = 1$). We remark that if in addition the underlying measure $\mu$ satisfies the Poincar\'e inequality, we can replace $\|f\|_r$ by $\|f\|_2$ and obtain an inequality for any $p \ge 2$ (with altered constants).

    If $r \in (1,2)$, one obtains
    \begin{align}\label{eq:Weibull_1}
    \|f\|_p \le e^{d}\|f\|_2 + \frac{2e}{r-1}(D^{1/r}\vee D^{1/2})\Big(p^{1/r^\ast}\bn |(\partial_i^\ast f)_{i=1}^{\lfloor p\rfloor}|_r\bn_p + p^{1/2}\bn |(\partial_i^\ast f)_{i=\lfloor p\rfloor + 1}^n|_2\bn_p\Big)\nonumber\\
    \end{align}
    for $p \ge 2$, where $\partial_1^\ast f(x),\ldots,\partial_n^\ast f(x)$ is the non-increasing rearrangement of the sequence $|\frac{\partial f(x)}{\partial x_1}|,\ldots,|\frac{\partial f(x)}{\partial x_n}|$.

    Note that
    \begin{displaymath}
    p^{1/r^\ast}\bn |(\partial_i^\ast f)_{i=1}^{\lfloor p\rfloor}|_r\bn_p + p^{1/2}\bn |(\partial_i^\ast f)_{i=\lfloor p\rfloor + 1}^n|_2\bn_p \le C p^{1/2}\bn |\nabla f|_2 \bn_p,
    \end{displaymath}
    so \eqref{eq:Weibull_1} is stronger then the bound \eqref{eq:classical} which has been derived from the classical logarithmic Sobolev inequality. Clearly to take advantage of the improvement one needs some additional information about the function $f$.

\end{enumerate}

\subsubsection{The non-defective case \label{sec:non-defective}}
Let us now pass to our second result, describing the moment estimates implied by modified LSI without defect, which will be later applied to obtain concentration bounds. Recall the definition of the inequality $mLSI(\Psi,D)$ given in formula \eqref{eq:mLSI}. The proofs of results presented in this section are deferred to Section \ref{sec:non-defective_proofs}.

Denote
\[
  L(K,D,\alpha,\beta) := \frac{1}{\alpha-1}(KD)^{1/\beta} + \Big(\frac{1}{\alpha-1} + \beta^{1/\alpha}\Big)(KD)^{1/\alpha}.
\]
\begin{theorem}\label{thm:Sobolev_centered}
Assume that $\Psi \colon \R^n \to \R$ satisfies the condition (C) and \eqref{ineq:growth-condition}
for some $1<\alpha \le 2\le \beta < \infty$.
Let $\mu$ be a probability measure on $\R^n$ satisfying $mLSI(\Psi,D)$. Then for all integrable (w.r.t. $\mu$) and locally Lipschitz functions $f\colon \R^n \to \R$ and all $p \ge \beta$,
\begin{align}\label{eq:mSobolev_centered}
\|f-\E_\mu f\|_p \le C L(K,D,\alpha,\beta) \bn |\nabla f|_{\Psi_p}\bn_p.
\end{align}
\end{theorem}

We note that, as is easy to see by truncation arguments, if the right-hand side of \eqref{eq:mSobolev_centered} is finite, then the function $f$ is $\mu$-integrable (in fact the $p$-th moment of $f$ is finite), so the integrability assumption is introduced in the above theorem just for formal reasons.

The advantage of \eqref{eq:mSobolev_centered} with respect to \eqref{eq:mSobolev} is that it provides estimates of central moments of $f$ in terms of norms of the gradient, without further dependence on any norms of $f$. This allows to derive concentration property for $f$ based on the regularity of the gradient and as a consequence provides also concentration bounds at the level of enlargements of sets.

Theorem \ref{thm:Sobolev_centered} is derived from Theorem \ref{thm:Sobolev} by means of Proposition \ref{prop:Poincare} below, which allows to handle the central moment of order $\beta$.

\begin{prop}\label{prop:Poincare}
Under the assumptions of Theorem \ref{thm:Sobolev_centered}, for every integrable (w.r.t. $\mu$) and locally Lipschitz function $f \colon \R^n \to \R$,
\begin{align}\label{eq:Poincare}
\|f - \E_\mu f\|_\beta \le C \Big( (KD)^{1/\beta} + (KD\beta)^{1/\alpha} \Big) \bn |\nabla f|_{\Psi_\beta}\bn_\beta.
\end{align}
\end{prop}

\subsection{Corollaries. Concentration: deviation inequalities and enlargement of sets.\label{sec:concentration}}
We will now explain how moment estimates of Theorem \ref{thm:Sobolev_centered} imply concentration results expressed in terms of non-Euclidean norms of the gradient and non-Euclidean enlargements of measurable sets. The proofs of results from this section are presented in Section \ref{sec:concentration_proofs}.

By Chebyshev's inequality we obtain the following
\begin{cor} \label{cor:Chebyshev}Under the assumptions of Theorem \ref{thm:Sobolev_centered}, for all integrable (w.r.t. $\mu$) and locally Lipschitz functions $f \colon \R^n \to \R$ and any $p \ge \beta$,
\begin{displaymath}
\mu\bigg(|f - \E_\mu f| \ge C L(K,D,\alpha,\beta)
    \bn|\nabla f|_{\Psi_p}\bn_p \bigg) \le e^{-p}.
\end{displaymath}
\end{cor}

The above corollary allows to get concentration bounds if one controls the growth of $g(p) := \bn|\nabla f|_{\Psi_p}\bn_p$, no Lipschitz-type conditions need to be assumed. However, since in the simplest situation one may control the growth of $g(p)$ via a uniform bound on $|\nabla f|_{\Psi_a}$ for some $a>0$ we will now specialize to functions which satisfy such a bound. To formulate the next corollary we will need to introduce the function $\omega_{\Psi}\colon \R_+\to \R_+ \cup \{\infty\}$, defined as
\begin{displaymath}
\omega_\Psi(t) = \sup_{x\in \R^n\setminus\{0\}, \; \Psi(x) \neq \infty} \frac{\Psi(tx)}{\Psi(x)}.
\end{displaymath}
If $\Psi$ satisfies the condition (C) then $\omega_\Psi$ is left-continuous, $\lim_{t \to 0} \omega_\Psi(t)=0$, $\lim_{t\to\infty} \omega_\Psi(t) = \infty$ and $t \mapsto \omega_\Psi(t)/t$ is non-decreasing, so $\omega_\Psi$ is strictly increasing on $\{ \omega_\Psi < \infty\}$. Consider the inverse $\omega_\Psi^{-1} \colon \R_+\to\R_+$ of $\omega_\Psi$, formally defined as
\[
  \omega_\Psi^{-1}(s) = \sup\{ t > 0 \colon \omega_\Psi(t) \le s\}.
\]
This function is continuous and since $t \mapsto \omega_\Psi(t)/t$ is non-decreasing, the function $s \mapsto s/\omega_\Psi^{-1}(s)$ is also continuous and non-decreasing.
If additionally $\Psi$ satisfies~\eqref{ineq:growth-condition} with some $K \ge 1$ and $1 < \alpha \le 2 \le \beta$ then for all $t > 0$,
\begin{align}\label{ineq:bounds-on-omega}
  K^{-1} (t^\alpha \land t^\beta) \le \omega_\Psi(t) \le K (t^\alpha \lor t^\beta)
\end{align}
which implies that
\begin{align}\label{omega-by-t}
  \lim_{t \to 0} \frac{\omega_\Psi(t)}{t} = 0, \qquad \lim_{t\to \infty} \frac{\omega_\Psi(t)}{t} = \infty
\end{align}
and in turn $s/\omega_\Psi^{-1}(s) \to 0$ as $s \to 0$ and $s/\omega_\Psi^{-1}(s) \to \infty$ as $s \to \infty$. Therefore one can define a function $\omega_\Psi^\ast \colon \R_+ \to \R_+$ to be a right-continuous inverse of $s \mapsto \frac{s}{\omega^{-1}_\Psi(s)}$, i.e.
\[
  \omega_\Psi^\ast(t) = \sup\Big\{ s > 0 \colon \frac{s}{\omega^{-1}_\Psi(s)} \le t\Big\}.
\]
Note that $\omega_\Psi^\ast(t)$ is strictly increasing.
We shall use the following observation (quite standard in the theory of Orlicz functions) which shows that the behaviour of the pair of functions $\omega_\Psi$ and $\omega_\Psi^\ast$ is similar to behaviour of conjugate functions:
\begin{lemma}\label{lemma:omega-omega-ast}
Assume that $\Psi$ satisfies the condition (C) and is such that \eqref{omega-by-t} holds. Then for any $t > 0$,
\begin{align}\label{omega-ast-equiv-def}
 \omega_\Psi^\ast(t) = t \sup\Big\{u>0 \colon \frac{\omega_\Psi(u)}{u} \le t\Big\}.
\end{align}
Moreover, if $\lambda(t) = \sup_{y>0}(ty - \omega_\Psi(y))$ is the Legendre transform of $\omega_\Psi$ then for all $t>0$,
\begin{align}\label{omega-ast-and-lambda}
  \lambda(t) \le \omega_\Psi^\ast(t) \le \lambda(2t).
\end{align}
\end{lemma}

The role played the function $\omega_\Psi^\ast$ in concentration inequalities is revealed by the following
\begin{cor} \label{cor:Lipschitz_conc}
Under the assumptions of Theorem \ref{thm:Sobolev_centered}, if a locally Lipschitz function $f\colon \R^n \to \R$ satisfies $|\nabla f(x)|_{\Psi_a} \le b$, $\mu$-a.e. for some $a,b > 0$, then for all $t > 0$,
\begin{displaymath}
\mu\Big(|f - \E_\mu f| \ge t\Big) \le  \exp\Big(\beta -a\omega_{\Psi}^\ast\Big(\frac{t}{CL(K,D,\alpha,\beta)b}\Big)\Big).
\end{displaymath}
\end{cor}

A version of the above corollary was obtained in \cite[Proposition 26]{MR2430612} in the special case when $\Psi(x) = \sum_{i=1}^n H(x_i)$ and $H$ is an even convex function on $\R$ such that $t \mapsto H(t)/t^2$ is non-decreasing on $(0,\infty)$. Our result and the result of \cite{MR2430612} are not directly comparable, on the one hand in \cite{MR2430612} there is no assumption concerning the parameter $\beta$ and the constants are explicit, on the other hand the argument there is restricted to functions $\Psi$ of a special form and to $\alpha = 2$, i.e. to the case of super-Gaussian tails. We remark that the proof in \cite{MR2430612} is based on the classical Herbst argument with the Laplace transform.

We will now express the concentration property of a measure $\mu$ satisfying $mLSI(\Psi,D)$ in the language of enlargements of sets. We will do it under an additional assumption that the function $\Psi$ is convex.
\begin{cor}\label{cor:enlargements} Assume that $\Psi$ is convex and let $\Psi^\ast$ be its Legendre transform. Under the assumptions of Theorem \ref{thm:Sobolev_centered}, for every Borel set $A \subseteq \R^n$ such that $\mu(A)\ge 1/2$ and every $u>0$,
\begin{displaymath}
\mu\Big(A +\{x\in\R^n\colon \Psi^\ast(x)< u\}\Big) \ge 1 - 2e^{\beta-u/(C(K,D,\alpha,\beta))}.
\end{displaymath}
\end{cor}

\paragraph{\bf Example} Consider $\Psi(x) = \sum_{i=1}^n (|x_i|^2\ind{|x_i|\le 1} + |x_i|^r\ind{|x_i| > 1})$ for some $r \in (1,\infty)$. For $r < 2$, the function $\Psi$ is not convex, but one can easily see that it is equivalent to a convex function, so we can still apply Corollary \ref{cor:enlargements} at the cost of adjusting the constants. In this case (after replacing $\Psi^\ast$ by an equivalent function) if one denotes by $B_p^n$ the unit ball of $\ell_p^n$, one obtains that for $r \ge 2$,
\begin{displaymath}
\p(A + \sqrt{u}B_2^n + u^{1/r^\ast}B_{r^\ast}^n) \ge 1 - C_{r,D}e^{-\frac{u}{C_{r,D}}}
\end{displaymath}
whereas for $r \in (1,2)$,
\begin{displaymath}
\p(A + (\sqrt{u}B_2^n) \cap(u^{1/r^\ast}B_{r^\ast}^n)) \ge 1 - C_{r,D}e^{-\frac{u}{C_{r,D}}}.
\end{displaymath}

In the case of the product distribution with marginal densities proportional to $e^{-|x_i|^{r^\ast}}$ the above inequalities were first obtained by Talagrand \cite{TalCan} (for $r \ge 1$). The first functional approach was proposed by Bobkov and Ledoux \cite{BobLed_exp} ($r = \infty$) and Barthe and Roberto \cite{MR2430612} ($r \in [2,\infty)$ as well as more general concentration rates between subexponential and sub-Gaussian), who used the modified log-Sobolev inequalities introduced by Gentil, Guillin and Miclo. A uniform setting for various types of concentration inequalities, including the ones mentioned above was proposed by Gozlan, who used Poincar\'e inequalities with modified norms of gradients \cite{MR2682264}. There are some subtle differences between the strength of various approaches, for instance Gozlan's approach works also for $ r = \infty$ and his constants do not depend on $r$. On the other hand in the non-product case his method introduces some dependence on the dimension $n$ (see e.g. Proposition 1.2. in \cite{MR2682264}).

\subsection{Further corollaries. Concentration inequalities for polynomials\label{sec:polynomials}}
In this section we will present corollaries concerning polynomial like functions. First we will consider homogeneous polynomials with coefficients in a Banach space, then arbitrary real valued polynomials or more generally functions with bounded derivatives of order $k$. The proofs of presented results are deferred to Section \ref{sec:polynomials_proofs}.

To formulate our results in a concise way we will need to introduce some additional notation. Namely for two $k$-indexed matrices $A = (a_{i_1,\ldots,i_k})_{i_1,\ldots,i_k=1}^n$ and
$B = (b_{i_1,\ldots,i_k})_{i_1,\ldots,i_k=1}^n$, where $a_{i_1,\ldots,i_k} \in E$ for some Banach space $E$ and $b_{i_1,\ldots,i_k} \in \R$ we set
\begin{displaymath}\langle A,B\rangle = \sum_{i_1,\ldots,i_k=1}^n a_{i_1,\ldots,i_k}b_{i_1,\ldots,i_k}.
 \end{displaymath}
Moreover, for vectors $x^1,\ldots,x^k \in \R^n$ we define $x^1\otimes\cdots\otimes x^k = (x^1_{i_1}\cdots x^k_{i_k})_{i_1,\ldots,i_k=1}^n$. With this convention, the $E$-valued homogeneous form of degree $k$, given by matrix $A$ as above, i.e.
\begin{displaymath}
\sum_{i_1,\ldots,i_k = 1}^n a_{i_1,\ldots,i_k}x_{i_1}\cdots x_{i_k}
\end{displaymath}
can be written simply as $\langle A,x^{\otimes k}\rangle$.

By $\D^k f$ we will denote the $k$-th derivative of a function $f\colon \R^n \to \R$, which we will identify with the corresponding $k$-indexed matrix of partial derivatives.

\subsubsection{Concentration for Banach space valued chaos}

Let $(E,|\cdot|_E)$ be a separable Banach space and $A = (a_{i_1,\ldots,i_k})_{i_1,\ldots,i_k=1}^n$ a $k$-indexed $E$-valued matrix and $X = (X_1,\ldots,X_n)$ a random vector in $\R^n$. We will consider the random variable $Z = |\langle A,X^{\otimes k}\rangle|_E$. Without loss of generality we will assume that $A$ is symmetric, i.e. for any permutation $\sigma$ of the set $\{1,\ldots,k\}$, $a_{i_1,\ldots,i_k} = a_{i_{\sigma(1)},\ldots,i_{\sigma(k)}}$.

Our main result is the following

\begin{theorem}\label{thm:Borell} Assume that $\Psi$ is a convex function satisfying the conditions (C) and $(G_{K,\alpha,\beta})$ and let $X$ be a random vector in $\R^n$, whose law is absolutely continuous and satisfies the $mLSI(\Psi,D)$. For any $p \ge \beta$,
\begin{align}\label{eq:Borell}
\|Z - \E Z\|_p \le C_{D,K,\alpha,\beta,k}\sum_{j=1}^k \E\sup_{y^1,\ldots,y^j \in A_{\Psi,p}}|\langle A,y^1\otimes\cdots\otimes y^j \otimes X^{\otimes(k-j)}\rangle|_E,
\end{align}
where $A_{\Psi,p} = \{x \in R^n\colon \Psi^\ast(x) \le p\}$. As a consequence, for any $p \ge \beta$,
\begin{displaymath}
\p\Big(|Z - \E Z| \ge C_{D,K,\alpha,\beta,k}\sum_{j=1}^k \E\sup_{y^1,\ldots,y^j \in A_{\Psi,p}}|\langle A,y^1\otimes\cdots\otimes y^j \otimes X^{\otimes(k-j)}\rangle|_E\Big) \le e^{-p}.
\end{displaymath}
\end{theorem}

Versions of the above theorem were first obtained for Gaussian vectors by Borell \cite{Bo} and Arcones-Gin\'e \cite{MR1201060}. Subsequently they were proved for $X$ with independent coordinates possessing log-concave tails by {\L}ochowski \cite{Loch}) and Adamczak \cite{AdLogSobConv}. However, they provided rather estimates of $\|Z\|_p$ and the deviation of $Z$ above $C\E Z$, then concentration around $\E Z$.

Let us analyze the quantities appearing on the right-hand side of \eqref{eq:Borell}. Except for the one corresponding to $i=k$, they are all expectations of suprema of random variables and as such are difficult to estimate. The exceptional term however is `deterministic' and it is easy to see that for $p \to \infty$ it dominates the whole sum. Estimates of this form may be therefore used to obtain some large deviation type estimates for $|Z - \E Z|$. Also, in certain situations estimates of the troublesome expectations are available. This is the case e.g. if $X$ is a Gaussian vector and $E$ is real \cite{L2} or more generally $E$ is a Hilbert space (this result is unpublished but it may be recovered from estimates in \cite{L2}), and also if $E = \R$, $X$ has independent coordinates with log-concave tails and $k \le 3$ \cite{L3, L1, Chaos3d}.

\paragraph{\bf Example} Let us illustrate Theorem \ref{thm:Borell} on a simple example of a real-valued quadratic form $Z = \sum_{i,j=1}^n a_{ij}X_iX_j$ in a centered random vector $X=(X_1,\ldots,X_n)$ whose law $\mu$ satisfies $mLSI(\Psi,D)$ with $\Psi(x) = |x|_q^q$ for some $q \in (1,2]$ (the case studied in \cite{MR1800062,MR2146071}).  We have $\Psi^\ast(x) = C_q |x|_{q^\ast}^{q^\ast}$, therefore we obtain

\begin{align*}
\|Z - \E Z\|_p & \le C_{D,q} \Big(p^{1/q^\ast} \E\sup_{y \in B_{q^\ast}^n} \Big|\sum_{i,j=1}^n a_{ij} y_i X_j\Big| + p^{2/q^\ast} \sup_{x,y\in B_{q^\ast}^n} \Big|\sum_{i,j=1}^n a_{ij} x_iy_j\Big|\Big)\\
&= C_{D,q} \Big(p^{1/q^\ast} \E\Big( \sum_{i=1}^n \Big|\sum_{j=1}^na_{ij} X_j\Big|^q\Big)^{1/q} + p^{2/q^\ast} \sup_{x,y\in B_{q^\ast}^n} \Big|\sum_{i,j=1}^n a_{ij} x_iy_j\Big|\Big)\\
& \le C_{D,q} \Big(p^{1/q^\ast} \Big( \sum_{i=1}^n \E \Big|\sum_{j=1}^na_{ij} X_j\Big|^q\Big)^{1/q} + p^{2/q^\ast} \sup_{x,y\in B_{q^\ast}^n} \Big|\sum_{i,j=1}^n a_{ij} x_iy_j\Big|\Big).
\end{align*}
Now, by Proposition \ref{prop:Poincare}, applied with $\beta = 2$, for each $i$ (note that $\Psi$ satisfies \eqref{ineq:growth-condition} with $K = 1, \alpha = q$ and $\beta = 2$),
\begin{displaymath}
\E \Big|\sum_{j=1}^na_{ij} X_j\Big|^q\le \bn \sum_{j=1}^n a_{ij} X_j\bn_2^q \le C_{D,q} |(a_{ij})_{j=1}^n|_{\Psi_2}^q = 2^{q/q^\ast} C_{D,q}|(a_{ij})_{j=1}^n|_q^q.
\end{displaymath}
Thus we obtain
\begin{displaymath}
\|Z - \E Z\|_p \le C_{D,q}(p^{1/q^\ast} A + p^{2/q^\ast} B),
\end{displaymath}
where
\begin{displaymath}
A = \Big(\sum_{i,j=1}^n |a_{ij}|^q\Big)^{1/q}, \quad B = \sup_{x,y \in B_{q^\ast}^n} \Big|\sum_{i,j=1}^n a_{ij}x_i y_j\Big|.
\end{displaymath}
As a consequence, for all $t \ge 0$,
\begin{align}\label{eq:q_Hanson-Wright}
\p(|Z - \E Z| \ge t) \le 2\exp\Big(-c_{D,q} \min\Big(\frac{t^{q^\ast}}{A^{q^\ast}},\frac{t^{q^\ast/2}}{B^{q^\ast/2}}\Big)\Big).
\end{align}

Clearly Theorem \ref{thm:Borell} may be applied also to quadratic forms or forms of higher order, with values in Banach spaces, but the resulting estimates will be then given in terms of expectations of suprema, which may not be so easy to estimate. Let us remark that inequalities of the form \eqref{eq:q_Hanson-Wright} with $q=2$ are known as Hanson-Wright inequalities. In a slightly weaker form they were first proven in \cite{HW} for quadratic forms in independent sub-Gaussian variables.

\subsubsection{Comparison principles for real-valued polynomials \label{sec:real_polynomials}}
We will now restrict to a special choice of the function $\Psi$ related to the study of moments of linear combinations of i.i.d. random variables with logarithmically concave tails. We will start with a brief description of the results by Gluskin-Kwapie\'n \cite{GK}.

\begin{theorem}\label{thm:GK}
Let $\Phi \colon \R_+ \to \R_+\cup\{\infty\}$ be a convex non-decreasing function, such that $\Phi(0) = 0$ and $\Phi(1) = 1$. Consider a sequence $Z_1,\ldots,Z_n$ of independent symmetric random variables
satisfying $\p(|Z_i| \ge t) = e^{-\Phi(t_i)}$. Define the functions $\tilde{\Phi}\colon \R_+ \to \R_+\cup\{\infty\}$,
\begin{displaymath}
\tilde{\Phi}(x) = \left\{\begin{array}{ccc}|x|^2&\textrm{if}& |x| \le 1\\
\Phi(x)&\textrm{if}& |x| \ge 1
\end{array}\right.
\end{displaymath}
and $\Psi\colon \R^n \to \R_+\cup\{\infty\}$,
\begin{displaymath}
\Psi(x) = \sum_{i=1}^n {\tilde{\Phi}}^\ast(x_i),
\end{displaymath}
where $\tilde{\Phi}^\ast$ is the Legendre transform of $\tilde{\Phi}$.

Then for every sequence $x_1,\ldots,x_n$ of real numbers and every $p \ge 2$,
\begin{displaymath}
\frac{1}{C}|x|_{\Psi_p} \le \Big\|\sum_{i=1}^n x_iZ_i\Big\|_p\le C|x|_{\Psi_p}.
\end{displaymath}
\end{theorem}

We remark that the assumption $\Phi(1) = 1$ is just a normalization condition which allows to obtain two-sided moment estimates with a universal constant $C$ (otherwise one would have to replace $C$ by some (explicit) constant $C_\Psi$).

As already mentioned in the introduction, modified log-Sobolev inequalities with the function $\Psi$ as in Theorem \ref{thm:GK} were introduced by Gentil-Guillin-Miclo \cite{MR2198019,MR2351133} and further studied by Barthe-Roberto \cite{MR2430612} (when $\Phi(x)/x^2$ is non-increasing, which corresponds to super-Gaussian tail behaviour) and Gentil \cite{MR2487856} (when $\Phi(x)/x^2$ is non-decreasing, which corresponds to sub-Gaussian tail behaviour).

In view of Theorem \ref{thm:GK}, our Theorem \ref{thm:Sobolev_centered} can be given an interpretation in terms of independent random variables.

\begin{cor}\label{cor:comparison}
Under the notation of Theorem \ref{thm:GK}, further assume that for some $K \ge 1$ and $1<\alpha \le 2\le \beta < \infty$,
\begin{align}\label{ineq:growth-of-Phi}
  K^{-1} t^{\beta^\ast} \le \frac{\tilde{\Phi}(tu)}{\tilde{\Phi}(u)} \le K t^{\alpha^\ast},
\end{align}
for all $t \ge 1$ and $u > 0$.
Assume that a measure $\mu$ on $\R^n$ satisfies the $mLSI(\Psi,D)$. Let $X$ be a random vector with law $\mu$ and a vector $Z = (Z_1,\ldots,Z_n)$ be a sequence of i.i.d. symmetric random variables, independent of $X$, such that $\p(|Z_i|\ge t) = e^{-\Phi(t)}$ for $t \ge 0$. Then for every locally Lipschitz function $f\colon \R^n \to \R$ and every $p \ge 2$,
\begin{displaymath}
\|f(X)-\E f(X)\|_p \le C(K,D,\alpha,\beta) \|\langle \nabla f(X),Z\rangle \|_p.
\end{displaymath}
\end{cor}

The interest in the above reformulation of moment inequalities stems from the fact that it can be used as a linearization tool, which allows to get estimates for functions with bounded-derivatives of higher order, in particular polynomials.

\begin{theorem}\label{thm:higher_derivatives} In the setting of Corollary \ref{cor:comparison}, let $Z^1,\ldots,Z^k$ be independent copies of $Z$, independent of $X$. Then for every function $f \colon \R^n \to \R$ of class $\mathcal{C}^k$ and every $p \ge 2$ we have
\begin{displaymath}
\|f(X)-\E f(X)\|_p \le C_{D,k}\Big(\|\langle \D^k f(X),Z^1\otimes\cdots\otimes Z^k\rangle \|_p +\sum_{i=1}^{k-1} \|\langle \E_X \D^i f(X),Z^1\otimes \cdots\otimes Z^i\rangle\|_p\Big).
\end{displaymath}
\end{theorem}

Note that all the terms on the right-hand side, except for the first one are moments of polynomials in independent random variables. This is also the case for the first term, provided that $f$ itself is a polynomial of degree $k$. One can thus think of Theorem \ref{thm:higher_derivatives} as a tool which allows to transfer estimates for polynomials in independent random variables to functions with bounded derivatives of higher order of random vectors $X$, whose law satisfies $mLSI(\Psi,D)$. We remark that there are many results concerning polynomials in independent random variables with log-concave tails, among available results there are hypercontractive estimates, two-sided estimates in terms of expected suprema of certain empirical processes (as in Theorem \ref{thm:Borell}) and in some cases (polynomials in Gaussian or exponential variables, polynomials in general variables with log-concave tails of degree at most 3) also precise two-sided inequalities in terms of `deterministic' quantities. We do not present the detailed discussion here, since it would require introducing rather technical notation and would anyway boil down to an application of known estimates. Instead in the example below we work out a simple application, again to a quadratic form.

\medskip

\paragraph{\bf Example:} Let $\Psi(x) = \sum_{i=1}^n(|x_i|^2\ind{|x_i|\le 1} + |x_i|^r\ind{|x_i|> 1})$ for some $r \ge 2$ and assume that $X = (X_1,\ldots,X_n)$ is a random vector whose law satisfies $mLSI(\Psi,D)$. For simplicity assume further that $X$ is centered. Consider finally a quadratic form $Y = f(X)$ for $f(x) = \sum_{i,j=1}^n a_{ij}x_ix_j$, where we assume without loss of generality that $a_{ij} = a_{ji}$. Thanks to centering, we have $ \E \nabla f(X) = 0$. Moreover $\D^2 f = (2a_{ij})_{i,j=1}^n$. Therefore, by Theorem \ref{thm:higher_derivatives}, if $Z_1,Z_1',\ldots,Z_n,Z_n'$ is a sequence of i.i.d. symmetric random variables, such that $\p(|Z_i|\ge t) = \exp(-t^{r^\ast})$, we get
\begin{displaymath}
\| Y - \E Y\|_p \le C_{D,r}\bn \sum_{i,j=1}^n a_{ij}Z_i Z_j'\bn_p
\end{displaymath}
for $p \ge 2$.

Using results from \cite{L1}, one can find a deterministic expression equivalent to the $p$-th moment on the right-hand side above. It is expressed in terms of certain norms of the matrix $A = (a_{ij})_{i,j=1}^n$, treated as a multi-linear functional on products of certain $\ell_2$ and $\ell_{r^\ast}$ spaces. More precisely,
\begin{align*}
\bn\sum_{i,j=1}^n a_{ij}Z_iZ_j'\bn_p  \simeq &p^{1/2}\|A\|_{\{1,2\}|\emptyset} + p\|A\|_{\{1\}\{2\}|\emptyset} + p^{1/r^\ast}\|A\|_{\emptyset|\{1,2\}}\\
 &+ p^{1/2+1/r^\ast}\|A\|_{\{1\}|\{2\}} + p^{2/r^\ast}\|A\|_{\emptyset|\{1\}\{2\}},
\end{align*}
where
\begin{align*}
\|A\|_{\{1,2\}|\emptyset} & = \sup\Big\{\sum_{i,j=1}^n a_{ij}x_{ij}\colon \sum_{i,j}^n x_{ij}^2 \le 1\Big\} = \Big(\sum_{i,j=1}^n a_{ij}^2\Big)^{1/2},\\
\|A\|_{\{1\}\{2\}|\emptyset} &= \sup\Big\{\sum_{i,j=1}^n a_{ij}x_{i}y_j\colon \sum_{i=1}^n x_{i}^2\le 1, \sum_{j=1}^n y_j^2 \le 1\Big\},\\
\|A\|_{\{1\}|\{2\}} & = \sup\Big\{\sum_{i,j=1}^n a_{ij}x_{i}y_j\colon \sum_{i=1}^n x_{i}^2\le 1, \sum_{j=1}^n |y_j|^{r^\ast} \le 1\Big\},\\
\|A\|_{\emptyset|\{1,2\}} & = \sup\Big\{\sum_{i,j=1}^n a_{ij}x_{ij}\colon \sum_{i,j=1}^n |x_{ij}|^{r^\ast} \le 1\Big\} = \Big(\sum_{i,j=1}^n |a_{ij}|^r\Big)^{1/r},\\
\|A\|_{\emptyset|\{1\}\{2\}} &= \sup\Big\{\sum_{i,j=1}^n a_{ij}x_{i}y_j\colon \sum_{i=1}^n |x_{i}|^{r^\ast}\le 1, \sum_{j=1}^n |y_j|^{r^\ast} \le 1\Big\}.\\
\end{align*}
As a consequence we obtain that for $p \ge 2$,
\begin{align*}
\| Y - \E Y\|_p \le & C_{D,r} \Big(p^{1/2}\|A\|_{\{1,2\}|\emptyset} + p\|A\|_{\{1\}\{2\}|\emptyset} + p^{1/r^\ast}\|A\|_{\emptyset|\{1,2\}}\\
 &+ p^{1/2+1/r^\ast}\|A\|_{\{1\}|\{2\}} + p^{2/r^\ast}\|A\|_{\emptyset|\{1\}\{2\}}\Big)
\end{align*}
and so for $t \ge 0$,
\begin{multline*}
\p(|Y - \E Y| \ge t) \\ \le 2\exp\Big(-c_{D,r} \min\Big(\frac{t^2}{\|A\|_{\{1,2\}|\emptyset}^2},\frac{t}{\|A\|_{\{1\}\{2\}|\emptyset}},
\frac{t^{r^\ast}}{\|A\|_{\emptyset|\{1,2\}}^{r^\ast}},\frac{t^{\frac{2r^\ast}{r^\ast+2}}}{\|A\|_{\{1\}|\{2\}}^{\frac{2r^\ast}{r^\ast+2}}},\frac{t^{r^\ast/2}}{\|A\|_{\emptyset|\{1\}\{2\}}^{r^\ast/2}}\Big)\Big).
\end{multline*}
In the class of random vectors satisfying $mLSI(\Psi,D)$ this estimate is optimal up to constants (as it can be reversed for the vector $Y = (Z_1,\ldots,Z_n)$). A similar derivation may be also carried out for cubic forms as two-sided estimates of their moments are known \cite{Chaos3d}, however it would involve 10 different norms of the corresponding $3$-indexed matrix (under the assumption that $X$ is isotropic). As for forms of higher order, they can also be reduced to forms in variables $Z_1,\ldots,Z_n$, by means of Theorem \ref{thm:higher_derivatives}. However finding two-sided estimates for moments of the latter forms remains open.

\subsection{Concentration results for functions with bounded Hessian under the logarithmic Sobolev inequality\label{sec:BCG}}

In this section we will consider the setting of the classical logarithmic Sobolev inequality and we  will prove a two-level concentration estimate for functions with bounded derivatives of second order, which slightly improves on the special $\mathcal{C}^2$ case of Theorem 1.2. in \cite{Nonlipschitz} and Theorem \ref{thm:higher_derivatives}.
Our approach is inspired by a very recent development by Bobkov, Chistyakov and G{\"o}tze \cite{BCG} who considered second order concentration on the sphere $S^{n-1}$. While the authors of \cite{BCG} were interested mostly in subexponential concentration, it turns out that using their approach one can also obtain two-level bounds. The goal of this section is to describe this derivation. Actually, for consistency with previous sections, we will consider a slightly more general setting and obtain inequalities in terms of moments, which allows to obtain concentration for functions with unbounded but controlled Hessian. It will be at a cost of deteriorating constants with respect to what can be obtained by working with Laplace transforms in the bounded Hessian case (as in \cite{BCG}).

Recall that for a matrix $A = (A_{i,j})_{i,j=1}^n$, by $\|A\|_{HS}$ we denote the Hilbert-Shmidt norm of $A$, whereas $\|A\|_{op}$ stands for the operator norm of $A$, i.e. $\|A\|_{HS} = \sqrt{\sum_{i,j=1}^n a_{ij}^2}$, $\|A\|_{op} = \sup_{|x|,|y|\le 1} \sum_{i,j=1}^n a_{ij}x_iy_j$.

The main result of this section is the following
\begin{theorem}\label{thm:BCG-C2}
Let $\mu$ be a probability measure on $\R^n$, such that for every $p \ge 2$ and every smooth function $f\colon \R^n \to \R$,
\begin{align}\label{eq:BCG_assumption}
\|f - \E_\mu f\|_p \le L\sqrt{p}\bn |\nabla f|_2\bn_p.
\end{align}
Let $f \colon \R^n \to \R$ be a function of class $\mathcal{C}^2$, such that the operator norm of $\D^2 f$ is uniformly bounded on $\R^n$. Then for every $t > 0$,
\begin{displaymath}
\mu(|f - \E_\mu f| \ge t ) \le e^2 \exp\Big(-\min\Big(\frac{t^2}{ a^2},\frac{t}{b}\Big)\Big),
\end{displaymath}
where
\begin{displaymath}
a^2 = 4e^2\Big(\sqrt{2}L^2\bn |\D^2 f|_{\textup{HS}}\bn_2 + L|\E_\mu \nabla f|_2\Big)^2,\quad b = 2e L^2 \bn |\D^2 f|_{\textup{op}}\bn_\infty.
\end{displaymath}
\end{theorem}

In fact we shall prove a more general result, from which the above theorem easily follows.

\begin{theorem}\label{thm:BCG}
Let $\mu$ be as in Theorem \ref{thm:BCG-C2}. Then for every $k \ge 2$, every $f \colon \R^n \to \R$  of class $\mathcal{C}^k$ and for every $p \ge 2$,
\begin{align}\label{eq:BCG_statement}
\|f - \E_\mu f\|_p & \le  L\sqrt{p}\E |\nabla f|_2  + L^2p\bn |\D^2 f|_{\textup{op}} \bn_p\\
&\le \sqrt{p}\Big(2^{(k-1)/2}L^k \bn |\D^k f|_2 \bn_2 + \sum_{m=1}^{k-1}2^{(m-1)/2}L^m |\E_\mu \D^m f|_2\Big) +L^2 p\bn |\D^2 f|_{\textup{op}} \bn_p,\nonumber
\end{align}
where for an $m$-indexed matrix $A = (a_{i_1,\ldots,i_m})_{i_1,\ldots,i_m=1}^n$ we denote $|A|_{2} = \sqrt{\sum_{i_1,\ldots,i_m = 1}^n a_{i_1,\ldots,i_m}^2}$ and $|\cdot|_{\textup{op}}$ is the operator norm of a (two-indexed) matrix.
\end{theorem}

The advantage of Theorem \ref{thm:BCG-C2} over the $\mathcal{C}^2$ case of Theorem 1.2. in \cite{Nonlipschitz} stems from the fact that in the latter instead of $\| |\D^2 f|_{\textup{HS}} \|_2$ one has $\| |\D^2 f|_{\textup{HS}} \|_p$. As a consequence, the tail bound obtained in \cite{Nonlipschitz} uses $ \||\D^2 f|_{\textup{HS}} \|_\infty$ instead of $ \||\D^2 f|_{\textup{HS}} \|_2$. On the other hand it is not clear to us whether Theorem \ref{thm:BCG} could lead to similar improvements of the results in \cite{Nonlipschitz} in the case of functions with bounded derivatives of order higher than 2, since instead of the term $\||D^2 f|_{\textup{op}}\|_p$ the bounds in \cite{Nonlipschitz} involve $|\E_\mu \D^2 f|_{\textup{op}}$ (at the cost of introducing some additional norms of higher order derivatives). We refer the Reader to \cite{Nonlipschitz} for the details.

\section{Proofs\label{sec:proofs}}
In the proofs we will drop the subscript $\mu$ and write simply $\E, \Ent$ for $\E_\mu$, $\Ent_\mu$.

\subsection{Proofs of results from Section \ref{sec:defective}}\label{sec:defectiv_proofs}
Let us first state without proof the following well-known lemma, which follows from the convexity of the function $p \mapsto \log \|X\|_{1/p}$.

\begin{lemma}\label{le:moment_comparison} If $X$ is a random variable, such that for some $p > q>0$ and $A \ge 1$, $\|X\|_p \le A\|X\|_q$, then for all $0 < r < q$,
\begin{displaymath}
\|X\|_{p} \le A^{\frac{(p-r)q}{(p-q)r}}\|X\|_{r}.
\end{displaymath}
\end{lemma}

We are also going to use the following observation on the norms $|\cdot|_{\Psi_p}$.
\begin{lemma}\label{lemma:property-of-Psi}
If $\Psi$ satisfies (C) and \eqref{ineq:growth-condition} then for any $x \in \R^n$,
\[
  \Psi(x) \le K (|x|_{\Psi}^\alpha + |x|_{\Psi}^\beta).
\]
\end{lemma}
\begin{proof}
First note that (C) implies that $\Psi(x/|x|_\Psi) \le 1$. If $|x|_\Psi \ge 1$, then
\[
  \Psi(x) = \Psi\left(|x|_\Psi \frac{x}{|x|_\Psi} \right) \le K |x|_\Psi^\beta,
\]
and if $|x|_\Psi \le 1$, then
\[
  1 \ge \Psi\left(\frac{x}{|x|_\Psi}\right) \ge K^{-1} |x|_\Psi^{-\alpha} \Psi(x).
\]
\end{proof}

We are now ready to prove Theorem \ref{thm:Sobolev}.

\begin{proof}[Proof of Theorem \ref{thm:Sobolev}]
Consider an arbitrary locally Lipschitz bounded positive function $f\colon \R^n \to \R$. Arguing as in the proof of Theorem 3.4. in \cite{Nonlipschitz}, we get
\begin{displaymath}
  \frac{d}{dt} (\E f^t )^{2/t} = \frac{2}{t^2} (\E f^t)^{\frac{2}{t} - 1}\Ent f^t.
\end{displaymath}
Thus by $dmLSI(\Psi, D, d)$ applied to the function $f^{t/2}$ and by Lemma~\ref{lemma:property-of-Psi} applied to $\Psi_p$ we have
\begin{align*}
\frac{d}{dt} \left( \E f^t \right)^{2/t} &\le
\frac{2D}{t^2} \left( \E f^t \right)^{\frac{2}{t} - 1}  \E f^t\Psi\Big(\frac{t\nabla f}{2f}\Big) + \frac{2d}{t^2}(\E f^t)^{2/t} \\
&= \frac{2Dp}{t^2} \|f\|_t^{2-t} \E f^t\Psi_p\Big(\frac{t\nabla f}{2pf}\Big) + \frac{2d}{t^2}(\E f^t)^{2/t} \\
&\le \frac{2^{1-\alpha}KD}{t^{2-\alpha} p^{\alpha-1}} \|f\|_t^{2-t} \E f^{t-\alpha} |\nabla f|_{\Psi_p}^\alpha +
     \frac{2^{1-\beta}KD}{t^{2-\beta} p^{\beta-1}} \|f\|_t^{2-t} \E f^{t-\beta} |\nabla f|_{\Psi_p}^\beta
     + \frac{2d}{t^2} \|f\|_t^2.
\end{align*}
Further denote $M = KD$.
Using H{\"o}lder's inequality with pairs of exponents $\frac{t}{t-\alpha}$, $\frac{t}{\alpha}$ and $\frac{t}{t-\beta}$, $\frac{t}{\beta}$, for $t \in (\beta, p)$ we have
\begin{align*}
  \frac{d}{dt} \|f\|_t^2 &\le
    \frac{2^{1-\alpha}M}{t^{2-\alpha} p^{\alpha-1}} \|f\|_t^{2-t} \|f\|_t^{t-\alpha} \||\nabla f|_{\Psi_p}\|_t^\alpha
    + \frac{2^{1-\beta}M}{t^{2-\beta} p^{\beta-1}} \|f\|_t^{2-t} \|f\|_t^{t-\beta} \||\nabla f|_{\Psi_p}\|_t^\beta
    + \frac{2d}{t^2} \|f\|_t^2 \\
    &\le \frac{2^{1-\alpha}M}{t^{2-\alpha} p^{\alpha-1}} \|f\|_t^{2-\alpha} \||\nabla f|_{\Psi_p}\|_p^\alpha
    + \frac{2^{1-\beta}M}{t^{2-\beta} p^{\beta-1}} \|f\|_t^{2-\beta} \||\nabla f|_{\Psi_p}\|_p^\beta
    + \frac{2d}{t^2} \|f\|_t^2.
\end{align*}
For $t \in [\beta, p]$ define
\[
  x(t) = \|f\|_t^2 / \bn |\nabla f|_{\Psi_p} \bn_p^{2}
\]
(note that by Condition (C), if the denominator above vanishes, then $f$ is constant and the theorem is trivially satisfied).
Clearly $x$ is non-decreasing and in the view of the above it satisfies
\begin{align}\label{ineq:diff-ineq-for-x}
  \frac{dx}{dt} \le M a(t) x^{1-\alpha/2} + M b(t) x^{1-\beta/2} + d c(t) x
\end{align}
for $t \in (\beta,p)$, where
\[
  a(t) = 2^{1-\alpha} t^{\alpha-2} p^{1-\alpha}, \qquad b(t) = 2^{1-\beta} t^{\beta-2} p^{1-\beta}, \qquad c(t) = \frac{2}{t^2}.
\]
Now, consider three cases:
\paragraph{\bf{Case 1:}} $x(p) \le \frac{1}{(\alpha-1)^2} \big(M^{2/\alpha} \lor M^{2/\beta}\big)$. In this case we simply have
\[
  \|f\|_p \le \frac{1}{\alpha-1} \big(M^{1/\alpha} \lor M^{1/\beta}\big) \bn |\nabla f|_{\Psi_p} \bn_p.
\]

\paragraph{\bf{Case 2:}} $x(\beta) \ge \frac{1}{(\alpha-1)^2} \big(M^{2/\alpha} \lor M^{2/\beta}\big)$. It is easy to check that for $t = \beta$ we have
\begin{align}\label{ineq:bounds-for-a-b}
  M x(t)^{1-\alpha/2} \le M^{1/\alpha} x(t)^{1/2}, \qquad M x(t)^{1-\beta/2} \le M^{1/\beta} x(t)^{1/2},
\end{align}
and since $x(t)$ is non-decreasing, we clearly have the above for all $t \in (\beta, p)$. Combining~\eqref{ineq:diff-ineq-for-x}  with~\eqref{ineq:bounds-for-a-b} yields
\[
  \frac{dx}{dt} \le \big(M^{1/\alpha} a(t) + M^{1/\beta} b(t)\big) x^{1/2} + d c(t) x.
\]
Substituting $y = x^{1/2}$ we get
\[
  \frac{dy}{dt} \le \frac12 \big(M^{1/\alpha} a(t) + M^{1/\beta} b(t)\big) + \frac12 d c(t) y,
\]
from which we easily obtain
\begin{align*}
  y(p) &\le y(\beta) e^{\frac{d}{2} \int_{\beta}^p c} + \frac12 \Big(\int_\beta^p \big(M^{1/\alpha} a(t) + M^{1/\beta} b(t)\big)
            e^{-\frac{d}{2}\int_\beta^t c} \, dt \Big) e^{\frac{d}{2} \int_\beta^p c} \\
       &\le y(\beta) e^{\frac{d}{\beta}} + \frac12 \bigg(\frac{M^{1/\alpha}}{\alpha-1} + \frac{M^{1/\beta}}{2^{\beta-1}(\beta-1)} \bigg) e^{\frac{d}{\beta}} \\
       &\le y(\beta) e^{\frac{d}{\beta}} + \frac{1}{\alpha-1} \big(M^{1/\alpha} \lor M^{1/\beta} \big) e^{\frac{d}{\beta}}.
\end{align*}
If $d \le \beta$ then the above yields
\[
  y(p) \le y(\beta) e^{d/\beta} + \frac{e}{\alpha-1}(M^{1/\alpha} \lor M^{1/\beta}),
\]
which means
\[
  \|f\|_p \le e^{d/\beta} \|f\|_\beta + \frac{e}{\alpha-1}(M^{1/\alpha} \lor M^{1/\beta}) \bn |\nabla f|_{\Psi_p}\bn_p,
\]
and if $d > \beta$ then using $y(\beta) \ge \frac{1}{\alpha-1}\big(M^{1/\alpha} \lor M^{1/\beta}\big)$,
\[
  y(p) \le 2 y(\beta) e^{d/\beta} \le e^{2d/\beta} y(\beta)
\]
hence
\[
  \|f\|_p \le e^{2d/\beta} \|f\|_\beta.
\]

\paragraph{\bf{Case 3:}} $x(t_0) = \frac{1}{(\alpha-1)^2} \big(M^{2/\alpha} \lor M^{2/\beta}\big)$ for some $t_0 \in (\beta,p)$. Arguing as in Case 2, for $y = x^{1/2}$ we have
\begin{align}\label{ineq:diff-ineq-for-y}
  y(p) &\le y(t_0) e^{\frac{d}{2} \int_{t_0}^p c} + \frac12 \Big(\int_{t_0}^p \big(M^{1/\alpha} a(t) + M^{1/\beta} b(t)\big)
            e^{-\frac{d}{2}\int_{t_0}^t c} \, dt \Big) e^{\frac{d}{2} \int_{t_0}^p c} \\
       &\le y(t_0) e^{d(\frac{1}{t_0}-\frac{1}{p})} + \frac{1}{\alpha-1} \big(M^{1/\alpha} \lor M^{1/\beta} \big) e^{d(\frac{1}{t_0}-\frac{1}{p})}.\nonumber
\end{align}
\paragraph{\bf{Case 3a:}} $t_0 \ge \frac{dp}{d+p}$. Then $\frac{1}{t_0} - \frac{1}{p} \le \frac{1}{d}$ and using the fact that $y(t_0) = \frac{1}{\alpha-1}\big(M^{1/\alpha} \lor M^{1/\beta})$, the inequality~\eqref{ineq:diff-ineq-for-y} gives
\[
  y(p) \le 2e \frac{1}{\alpha-1}\big(M^{1/\alpha} \lor M^{1/\beta}\big),
\]
hence
\[
  \|f\|_p \le 2e \frac{1}{\alpha-1}\big(M^{1/\alpha} \lor M^{1/\beta}\big) \bn |\nabla f|_{\Psi_p} \bn_p.
\]
\paragraph{\bf{Case 3b:}} $t_0 < \frac{dp}{d+p}$. Again, using the fact that $y(t_0) = \frac{1}{\alpha-1}\big(M^{1/\alpha} \lor M^{1/\beta}\big)$ the inequality~\eqref{ineq:diff-ineq-for-y} implies
\begin{align*}
  y(p) &\le y(t_0) e^{d(\frac{1}{t_0}-\frac{1}{p})} + \frac{1}{\alpha-1} (M^{1/\alpha} \lor M^{1/\beta}) e^{d(\frac{1}{t_0}-\frac{1}{p})} \\
       &= 2 y(t_0)e^{d(\frac{1}{t_0}-\frac{1}{p})} \le y(t_0) e^{2d(\frac{1}{t_0}-\frac{1}{p})},
\end{align*}
which means
\[
  \|f\|_p \le A \|f\|_{t_0},
\]
with $A = e^{2d(\frac{1}{t_0}-\frac{1}{p})}$. Using Lemma~\ref{le:moment_comparison} with $q=t_0$ and $r = \beta$ we obtain
\begin{align*}
  \|f\|_p &\le A^{\frac{(p-\beta)t_0}{(p-t_0)\beta}} \|f\|_\beta = e^{2d \frac{p-t_0}{t_0 p} \frac{(p-\beta)t_0}{(p-t_0)\beta}} \|f\|_\beta \\
    &\le e^{2d/\beta} \|f\|_\beta.
\end{align*}

This ends the proof for bounded positive functions. Let us now assume that $f\colon \R^n \to \R$ is a bounded locally Lipschitz function. Set $g_m = |f| + 1/m$ for $m =1,2,\ldots$ and note that almost everywhere with respect to the Lebesgue measure, $f$ and all the functions $g_m$ are differentiable, moreover
$\nabla g_m(x) \neq 0$ implies that $\nabla f(x) = \nabla g_m(x)$. Thus the inequality for $f$ follows by a limiting argument. Removing the boundedness assumption is straightforward by a truncation argument.
\end{proof}

Let us now pass to the sketch of the proof of Proposition \ref{prop:alpha_equal_1}.

\begin{proof}[Proof of Proposition \ref{prop:alpha_equal_1}]
It is enough to follow the steps of Theorem \ref{thm:Sobolev} and to replace the splitting value $\frac{1}{(\alpha-1)^2} \big(M^{2/\alpha} \lor M^{2/\beta}\big)$ for $x(t)$ with $\big(D^2 \lor (KD)^{2/\beta}\big) (\log p)^2$.
\end{proof}

We will now provide the proof of Proposition \ref{prop:mLSI-for-nu}, which shows that the lower bound in our assumption \eqref{ineq:growth-condition} cannot be avoided.

\begin{proof}[Proof of Proposition~\ref{prop:mLSI-for-nu}]
Let us note that for $\Psi(x) = |x|$, $mLSI(\Psi, 2)$ is equivalent to that for all bounded, locally Lipschitz functions $f \ge 0$,
\begin{equation}\label{ineq:mLSI-for-nu}
  \Ent_\nu f \le \E_\nu |f'|.
\end{equation}

For a Borel set $A \subseteq \R$, denote $\nu^+(A) = \liminf_{\varepsilon\to 0} \frac{\nu(A + (-\varepsilon,\varepsilon)) - \nu(A)}{\varepsilon}$.
It is a general fact, based on the co-area formula, that the isoperimetric inequality
\begin{equation}\label{ineq:isoperimetry-nu}
  \nu(A) \log\frac{1}{\nu(A)} \lor (1-\nu(A)) \log\frac{1}{1-\nu(A)} \le \nu^+(A)
\end{equation}
valid for all Borel sets $A \subseteq \R$ implies~\eqref{ineq:mLSI-for-nu}. Indeed, on the one hand, using the variational formula for the entropy,
\begin{equation}\label{eq:var-formula-for-ent}
  \Ent_\nu f = \sup\Big\{ \int f g \, d\nu \colon g \colon \R \to \R \text{ is measurable, bounded and } \int e^g \, d\nu \le 1 \Big\},
\end{equation}
and the Fubini theorem, for any measurable and bounded $g \colon \R \to \R$ with $\int e^g \, d\nu \le 1$ and $\rho$ defined as a finite (signed) measure on $\R$ such that $d\rho = g d\nu$ we obtain
\begin{align*}
  \int fg \,d\nu &= \int f \, d\rho = \int_0^\infty \rho(\{f > t\}) \, dt = \int_0^\infty \int_\R \ind{f(x) > t} g(x) \,\nu(dx) \, dt \\
  &\le \int_0^\infty \Ent_\nu(\ind{f > t}) \, dt = \int_0^\infty \nu(\{f > t\})\log\frac{1}{\nu(\{f > t\})} \, dt,
\end{align*}
where the inequality follows from~\eqref{eq:var-formula-for-ent}. Hence,
\[
  \Ent_\nu f \le \int_0^\infty \nu(\{f > t\})\log\frac{1}{\nu(\{f > t\})} \, dt.
\]
On the other hand, by the co-area formula (see e.g. Theorem 8.5.1. in \cite{BLG}),
\[
  \E_\nu |f'| \ge \int_0^\infty \nu^+(\{f > t\}) \,dt
\]
which combined with the previous formula shows the implication \eqref{ineq:isoperimetry-nu} $\implies$ \eqref{ineq:mLSI-for-nu}.

For the isoperimetric inequality~\eqref{ineq:isoperimetry-nu} itself, since the density of $\nu$ w.r.t. the Lebesgue measure,
\[
  f_\nu(x) = F_\nu'(x) = \begin{cases}
    \frac12 e^{-(e^{-x} - (1-x))}, & \text{for $x<0$,}\\
    \frac12 e^{-(e^x - (1+x))}, &\text{for $x \ge 0$}
  \end{cases}
\]
is log-concave, the result of Bobkov~\cite{Bobkov-Extremal-log-concave} asserts that it is enough to check~\eqref{ineq:isoperimetry-nu} for half-lines, and in fact, by symmetry of $\nu$, for $A = [x,\infty)$ with $x \ge 0$.

Let $t \in (0,1/2]$. Then $F_\nu^{-1}(1-t) = \log(1+\log\frac{1}{2t})$ and thus
\[
  f_\nu(F_\nu^{-1}(1-t)) = t(1+\log\frac{1}{2t}) \ge t \log\frac{1}{t} = t\log\frac{1}{t} \vee (1-t)\log\frac{1}{1-t},
\]
which proves the isoperimetric inequality~\eqref{ineq:isoperimetry-nu}.

For the moment estimate~\eqref{ineq:moments-estimate-for-nu} one can repeat the argument from the proof of Theorem~\ref{thm:Sobolev} to get that the function $x(t)$ as defined therein satisfies
\[
  \frac{dx}{dt} \le \frac{2}{t} x^{1/2}(t),
\]
for $t \in (1,p)$, i.e. $(x^{1/2})' \le \frac{1}{t}$ and thus
\[
  x^{1/2}(p) - x^{1/2}(1) \le \log p,
\]
which implies \eqref{ineq:moments-estimate-for-nu}.
In order to show the `moreover' part of the proposition note that
\begin{align*}
  \|x\|_{L^p(\nu)} &= \Big( e \int_0^\infty p x^{p-1} e^{-e^x} \, dx \Big)^{1/p} \ge \Big( ep \int_{\frac12 \log p}^{\log p} \big(\frac12 \log p\big)^{p-1} e^{-e^{\log p}} \, dx \Big)^{1/p} \\
  &= \Big( ep \big(\frac12 \log p\big)^p e^{-p} \Big)^{1/p} \ge \frac{\log p}{2e}.
\end{align*}

\end{proof}

\subsection{Proofs of results from Section \ref{sec:non-defective}}\label{sec:non-defective_proofs}

Let us first prove the Poincar\'e type inequality given in Proposition \ref{prop:Poincare}.

\begin{proof}[Proof of Proposition \ref{prop:Poincare}] The proof follows the approach by Bobkov and Zegarli\'nski~\cite{MR2146071} who considered the function $\Psi(x) = |x|^q$ for $q \in [1,2]$.
Let $M f$ be the median of $f$ under $\mu$.
We have $\|f - \E f\|_\beta \le 2\|f - Mf\|_\beta$ so it is enough to prove \eqref{eq:Poincare} with the mean replaced by the median. In what follows without loss of generality we will assume that $M f = 0$.

Note that for any bounded, locally Lipschitz function $g \colon \R^n \to (0,\infty)$, the inequality $mLSI(\Psi,D)$ applied to $g^{\beta/2}$ and Lemma~\ref{lemma:property-of-Psi} yield
\begin{align*}
  \Ent g^\beta &\le D \E g^\beta \Psi\Big(\frac{\beta \nabla g}{2 g}\Big)
                  = D \beta \E g^\beta \Psi_{\beta}\Big(\frac{\nabla g}{2 g}\Big) \\
               &\le KD\beta \Big( 2^{-\beta} \E |\nabla g|_{\Psi_\beta}^\beta + 2^{-\alpha} \E g^{\beta-\alpha} |\nabla g|_{\Psi_\beta}^\alpha \Big).
\end{align*}
If $\alpha < \beta$ we apply the Young inequality to the last expectation to get
\[
  \E g^{\beta-\alpha} |\nabla g|_{\Psi_\beta}^\alpha \le A^{\beta/\alpha} \frac{\alpha}{\beta} \E |\nabla g|_{\Psi_\beta}^\beta + A^{-\beta/(\beta-\alpha)} \frac{\beta-\alpha}{\beta} \E g^\beta,
\]
and the choice $A = \big(2^{1-\alpha} KD(\beta-\alpha)\big)^{(\beta-\alpha)/\beta}$ yields
\begin{align}\label{eq:entropy_of_power}
  \Ent g^\beta &\le KD\beta \Big( 2^{-\beta} + 2^{-\alpha} \frac{\alpha}{\beta} \big(2^{1-\alpha} KD(\beta-\alpha)\big)^{(\beta-\alpha)/\alpha} \Big) \E |\nabla g|_{\Psi_\beta}^\beta + \frac12 \E g^\beta\nonumber \\
  &\le C\Big((KD)^{1/\beta} + (\beta KD)^{1/\alpha}\Big)^\beta \E |\nabla g|_{\Psi_\beta}^\beta + \frac12 \E g^\beta.
\end{align}
Note that~\eqref{eq:entropy_of_power} is obviously valid also in the case $\alpha=\beta$.
Also, this inequality can be extended to arbitrary non-negative locally Lipschitz function $g \colon \R^n \to \R$.
Moreover, by Lemma 2.2 in~\cite{MR2146071}, for any non-negative $h \colon \R^n \to \R$,
\begin{displaymath}
\Ent h \ge \Big(\log \frac{1}{\mu(\{h > 0\})}\Big) \E h,
\end{displaymath}
which used for $h = g^\beta$ and combined with~\eqref{eq:entropy_of_power} gives that for any non-negative locally Lipschitz function $g \colon \R^n \to \R$,
\begin{align}\label{eq:moment-gradient}
  \Big(\log \frac{1}{\mu(\{g > 0\})} - \frac12 \Big)\E g^\beta \le C\Big((KD)^{1/\beta} + (\beta KD)^{1/\alpha}\Big)^\beta \E |\nabla g|_{\Psi_\beta}^\beta.
\end{align}

Now, applying~\eqref{eq:moment-gradient} to the functions $g = f_+$ and $g = f_-$ and using the assumption $M f = 0$ and the implication $\nabla f_{\pm} \neq 0 \implies \nabla f_\pm = \pm \nabla f$, we obtain
\begin{align*}
(\log 2 - 1/2) \E f_+^\beta &\le C\Big((KD)^{1/\beta} + (\beta KD)^{1/\alpha}\Big)^\beta \E |\nabla f_+|_{\Psi_\beta}^\beta \\
&\le C\Big((KD)^{1/\beta} + (\beta KD)^{1/\alpha}\Big)^\beta \E |\nabla f|_{\Psi_\beta}^\beta\\
\end{align*}
and
\begin{align*}
(\log 2 - 1/2) \E f_-^\beta &\le C\Big((KD)^{1/\beta} + (\beta KD)^{1/\alpha}\Big)^\beta \E |\nabla f_-|_{\Psi_\beta}^\beta \\&\le C\Big((KD)^{1/\beta} + (\beta KD)^{1/\alpha}\Big)^\beta \E |\nabla f|_{\Psi_\beta}^\beta.
\end{align*}
Summing the above inequalities we get
\[
  \frac{1}{10} \E |f|^\beta \le 2 C\Big((KD)^{1/\beta} + (\beta KD)^{1/\alpha}\Big)^\beta  \E |\nabla f|_{\Psi_\beta}^\beta,
\]
which ends the proof.
\end{proof}

Having proven Proposition \ref{prop:Poincare} we can reduce Theorem \ref{thm:Sobolev_centered} to Theorem \ref{thm:Sobolev}.

\begin{proof}[Proof of Theorem \ref{thm:Sobolev_centered}] We apply Theorem \ref{thm:Sobolev} to the function $|f - \E f|$ to get for $p\ge \beta$,
\begin{displaymath}
\|f-\E f\|_p \le \|f-\E f\|_\beta + \frac{2e}{\alpha-1}\Big((KD)^{1/\alpha} + (KD)^{1/\beta}\Big) \bn|\nabla f|_{\Psi_p}\bn_p.
\end{displaymath}
Thus by Proposition \ref{prop:Poincare}, \eqref{ineq:monotonicity} and H\"older's inequality we obtain
\begin{align*}
\|f- \E f\|_p &\le C \Big( (KD)^{1/\beta} + (KD\beta)^{1/\alpha} \Big) \bn |\nabla f|_{\Psi_\beta}\bn_\beta +
  \frac{C}{\alpha-1}\Big((KD)^{1/\alpha} + (KD)^{1/\beta}\Big) \bn|\nabla f|_{\Psi_p}\bn_p\Big) \\
  &\le C\Big(\frac{1}{\alpha-1}(KD)^{1/\beta} + \Big(\frac{1}{\alpha-1} + \beta^{1/\alpha}\Big)(KD)^{1/\alpha} \Big)
    \bn|\nabla f|_{\Psi_p}\bn_p.
\end{align*}
\end{proof}

\subsection{Proofs of results from Section \ref{sec:concentration}}\label{sec:concentration_proofs}
Let us start with the proof of Lemma \ref{lemma:omega-omega-ast}.

\begin{proof}[Proof of Lemma~\ref{lemma:omega-omega-ast}]
The definition of $\omega^\ast_\Psi$ can be written equivalently as
\begin{align*}
  \omega^\ast_\Psi(t) = t \sup\big\{ u>0 \colon \omega_\Psi^{-1}(tu) \ge u \big\}.
\end{align*}
Now recall that if $f \colon \R_+ \to \R_+$ is left-continuous, non-decreasing and satisfies $\lim_{x \to 0} f(x) = 0$, $\lim_{x \to \infty} f(x) = \infty$
and $g \colon \R_+\to \R_+$ is a right-continuous inverse of $f$, i.e.
\[
  g(y) = \sup\{ x > 0 \colon f(x) \le y\},
\]
then for all $x, y > 0$,
\begin{align}\label{ineq:f-generalized-inverse}
  g(y) \ge x \quad \iff \quad f(x) \le y.
\end{align}
Applying~\eqref{ineq:f-generalized-inverse} with $\omega_\Psi$ as $f$ we obtain that $\omega_\Psi^{-1}(tu) \ge u$ if and only if $\omega_\Psi(u) \le tu$, which proves~\eqref{omega-ast-equiv-def}.

For the first inequality in~\eqref{omega-ast-and-lambda}, fix $t > 0$ and take $u > 0$ for which the supremum in~\eqref{omega-ast-equiv-def} is attained (such $u$ exists due to~\eqref{omega-by-t} and left-continuity of $\omega_\Psi$). Then for all $y > u$,
\[
  t < \frac{\omega_\Psi(y)}{y} \le \frac{\omega_\Psi(y)}{y-u},
\]
hence for all $y > u$,
\[
  ty - \omega_\Psi(y) \le tu.
\]
Moreover the above inequality holds trivially for $y \in (0,u]$. Therefore $\lambda(t) \le tu = \omega_\Psi^\ast(t)$.

To prove the second inequality of~\eqref{omega-ast-and-lambda} fix any $u > 0$ satisfying $\omega_\Psi(u)/u \le t$ and note that
\[
  tu \le 2tu - \omega_\Psi(u) \le \sup_{y>0} (2ty - \omega_\Psi(y)) = \lambda(2t).
\]
\end{proof}

Next, we will prove Corollary \ref{cor:Lipschitz_conc}.

\begin{proof}[Proof of Corollary \ref{cor:Lipschitz_conc}] Denote $L := C L(K,D,\alpha,\beta)$. In view of Corollary \ref{cor:Chebyshev} it is enough to show that for all $t, p > 0$,
\begin{align}\label{ineq:rel-p-t}
  p \le a\omega_\Psi^\ast(t/(Lb)) \quad \implies \quad |\nabla f(x)|_{\Psi_p} \le t/L,\; \mu\textrm{-a.e.}
\end{align}
Then for a given $t>0$ take $p = a\omega_\Psi^\ast(t/(Lb))$ and use Corollary~\ref{cor:Chebyshev} and~\eqref{ineq:rel-p-t} to obtain
\begin{align*}
  \mu\Big(|f-\E f| \ge t\Big) &\le \mu\Big(|f-\E f| \ge L \bn |\nabla f|_{\Psi_p} \bn_p \Big) \\
     &\le e^{-p} \ind{p \ge \beta} + \ind{p < \beta} \le e^{\beta - p} = e^{\beta - a\omega_\Psi^\ast(t/(Lb))}.
\end{align*}

Note that the hypothesis $|\nabla f(x)|_{\Psi_a} \le b,\; \mu\textrm{-a.e.}$ implies that for any $t, p > 0$,
\begin{align*}
\Psi_p\Big(\frac{L\nabla f(x)}{t}\Big) &= \frac{1}{p}\Psi\Big(\frac{pL\nabla f(x)}{t}\Big) \le \frac{a}{p}\frac{1}{a}\Psi\Big(\frac{a\nabla f(x)}{b}\Big)\omega_\Psi\Big(\frac{Lpb}{at}\Big)\\
&\le \frac{a}{p}\omega_\Psi\Big(\frac{Lpb}{at}\Big). 
\end{align*}
From~\eqref{omega-ast-equiv-def} in Lemma~\ref{lemma:omega-omega-ast} we have
\[
  \omega_\Psi\Big(\frac{Lb}{t}\frac{p}{a}\Big) \le \frac{p}{a} \quad\iff\quad \omega_\Psi^\ast\Big(\frac{t}{Lb}\Big) \ge \frac{p}{a},
\]
thus~\eqref{ineq:rel-p-t} follows.
\end{proof}

Before we prove Corollary \ref{cor:enlargements}, let us formulate a simple lemma.
\begin{lemma}\label{le:median}
Under the assumptions of Corollary \ref{cor:Lipschitz_conc} we have
\begin{displaymath}
\mu\Big(|f - M_\mu f| \ge t\Big) \le  2\exp\Big(\beta -a\omega_{\Psi}^\ast\Big(\frac{t}{2Lb}\Big)\Big),
\end{displaymath}
where $L = C L(K,D,\alpha,\beta)$ and $C$ is the constant from Corollary~\ref{cor:Lipschitz_conc}.
\end{lemma}
\begin{proof}
Since $\omega_\Psi^{-1}$ is increasing and right-continuous, we can take $t_0$ to be a smallest positive real satisfying
\[
  \omega^\ast\Big(\frac{t_0}{Lb}\Big) \ge \frac{\beta+\log 2}{a},
\]
or equivalently, $\exp\big(\beta- a \omega_\Psi^\ast(t/(Lb)\big) \le 1/2$.
Then by Corollary~\ref{cor:Lipschitz_conc}, $\mu\big(|f - \E f| \ge t_0\big) \le 1/2$ and thus $|M f - \E f| \le t_0$. Therefore, using Corollary~\ref{cor:Lipschitz_conc} for $t \ge 2 t_0$,
\[
\mu\Big(|f - M f| \ge t\Big) \le \mu\Big(|f - \E f| \ge \frac{t}{2}\Big) \le \exp\Big(\beta -a\omega_{\Psi}^\ast\Big(\frac{t}{2Lb}\Big)\Big).
\]
On the other hand if $t < 2t_0$ then by the definition of $t_0$,
\begin{displaymath}
2\exp\Big(\beta -a\omega_{\Psi}^\ast\Big(\frac{t}{2Lb}\Big)\Big) > 1,
\end{displaymath}
so the inequality of the lemma holds trivially.
\end{proof}

\begin{proof}[Proof of Corollary \ref{cor:enlargements}]
We recall that if $\Phi$ is a Young function on $\R^n$, $\Phi^\ast$ is the Legendre transform of $\Phi$ and $|\cdot|_\Phi^\ast$ denotes the norm on $\R^n$, dual to $|\cdot|_{\Phi}$, then for all $x \in \R^n$,
\begin{align}\label{ineq:comparison-dual-norms}
|x|_{\Phi^\ast} \le |x|_\Phi^\ast \le 2|x|_{\Phi^\ast}.
\end{align}

Using the lower bound from~\eqref{ineq:growth-condition}, one can show that $\Psi^\ast(x) = 0$ iff $x=0$ and $\Psi^\ast(x) < \infty$ at all $x$. Further, notice that $\{x\colon \Psi^\ast(x) < u\} = B(|\cdot|_{u^{-1}\Psi^\ast},1) :=\{x\colon |x|_{u^{-1}\Psi^\ast} < 1\}$. Define the function
\begin{displaymath}
f(x) = \inf_{y \in A}|x-y|_{u^{-1}\Psi^\ast}.
\end{displaymath}
The function $f$ is 1-Lipschitz with respect to the norm $|\cdot|_{u^{-1}\Psi^\ast}$, which implies that $\nabla f$ exists almost everywhere and $|\nabla f|_{u^{-1}\Psi^\ast}^\ast \le 1$. Since $u^{-1}\Psi^\ast = (\Psi_u)^\ast$, \eqref{ineq:comparison-dual-norms} implies that $|\nabla f|_{\Psi_u} \le 1$.

Therefore, by Lemma \ref{le:median}, applied with $a = u$, $b=1$, $t = 1$,
\begin{displaymath}
\mu\Big(f \ge Mf + 1\Big)\le 2\exp\Big(\beta - u\omega_\Psi^\ast(1/L)\Big),
\end{displaymath}
where $L = C L(K,D,\alpha,\beta)$.
By the assumption $\mu(A) \ge 1/2$, we have $M f = 0$ and so
\[
  A+\{\Psi^\ast(x)< u\} = \{f < 1\} = \{f < Mf + 1\},
\]
therefore the above inequality yields
\begin{displaymath}
\mu\Big(A+\{\Psi^\ast(x)< u\}\Big) \ge 1 - 2\exp\Big(\beta - u\omega_\Psi^\ast(1/L)\Big).
\end{displaymath}
It remains to bound $\omega_\Psi^\ast(1/L)$ from below.
Note that by the convexity of $\Psi$ and \eqref{ineq:bounds-on-omega}, the function $\omega_\Psi$ is convex and everywhere finite, hence continuous. Thus by Lemma~\ref{lemma:omega-omega-ast},
\begin{align}\label{ineq:omega-1-L}
  \omega^\ast_\Psi(1/L) = w/L,
\end{align}
where $w$ is such that $\omega_\Psi(w)/w = 1/L$. On the other hand, the upper bound in~\eqref{ineq:bounds-on-omega} yields
\[
  \frac{1}{L} = \frac{\omega_\Psi(w)}{w} \le K (w^{\alpha-1} \lor w^{\beta-1}),
\]
hence
\[
  w \ge \frac{1}{(KL)^{1/(\alpha-1)}} \land \frac{1}{(KL)^{1/(\beta-1)}} \ge \frac{1}{(KL)^{1/(\alpha-1)}} \land 1.
\]
Combining the above lower bound on $w$ with~\eqref{ineq:omega-1-L} we obtain
\[
  \omega^\ast_\Psi(1/L) \ge (K^{-1/(\alpha-1)} L^{-\alpha/(\alpha-1)}) \land L^{-1}.
\]
\end{proof}

\subsection{Proofs of results from Section \ref{sec:polynomials}}\label{sec:polynomials_proofs}

Let us start with the proofs of results concerning Banach space valued homogeneous polynomials.

\begin{proof}[Proof of Theorem \ref{thm:Borell}]
We will proceed by induction on $k$. Note that the norm $|\cdot|_E$ can be expressed as a supremum of countably many functionals, therefore it is enough to prove the theorem for $(E,|\cdot|) = (\ell_\infty^N, |\cdot|_\infty)$ with arbitrarily large, finite $N$. We will denote the index related to the $\ell_\infty^N$ structure in the superscript, i.e.
$a_{i_1,\ldots,i_k} = (a^1_{i_1,\ldots,i_k},\ldots,a^N_{i_1,\ldots,i_k})$ and $A^r = (a^r_{i_1,\ldots,i_k})_{i_1,\ldots,i_k=1}^n$.

For $k = 1$, the inequality in question is then just a dual formulation of Theorem \ref{thm:Sobolev_centered} for $f(x) = \max_{r\le N} |\sum_{i=1}^n a^r_i x_i|$. Indeed, it is enough to note that $A_{\Psi,p}$ is the unit ball for the Orlicz norm corresponding to the function $\frac{1}{p}\Psi^\ast$, which up to universal multiplicative constants is equivalent to the dual norm for $|\cdot|_{\Psi_p}$.

Let us thus assume that the theorem is true for all homogeneous forms of degree strictly smaller than $k$. We can assume that the real-valued matrices $\pm (a^r_{i_1,\ldots,i_k})$ are pairwise distinct and nonzero for $1\le r \le N$. Since the set of zeros of a non-trivial polynomial is Lebesgue null, there exist open sets $B_r^\varepsilon$, $r = 1,\ldots,N$, $\varepsilon = \pm 1$, such that $\R^n\setminus (\bigcup_{r=1}^n\bigcup_{\varepsilon=\pm 1} B_r^\varepsilon)$ is Lebesgue null and on $B_r^\varepsilon$,
\begin{displaymath}
|\langle A,x^{\otimes k}\rangle|_E  = \varepsilon \langle A^r,x^{\otimes k}\rangle.
\end{displaymath}
Denoting $[k] = \{1,\ldots,k\}$, we get that almost everywhere on $B_r^\varepsilon$,
\begin{align*}
\nabla f (x) & = \varepsilon \Big(\sum_{i_1,\ldots,i_k=1}^n \Big(a^r_{i_1,\ldots,i_k} \sum_{s=1}^k \ind{i_s = j}\prod_{\stackrel{1\le u\le k}{u\neq s}}x_{i_s}\Big)\Big)_{j=1}^n\\
& = \varepsilon \sum_{s=1}^k \Big(\sum_{(i_m)_{m \in [k]\setminus \{s\}} \in [n]^{k-1}} \Big(a^r_{i_1,\ldots,,i_{s-1},j,i_{s+1},\ldots,i_k} \prod_{\stackrel{1\le u\le k}{u\neq s}}x_{i_s}\Big)\Big)_{j=1}^n\\
& = k\varepsilon \Big(\sum_{i_2,\ldots,i_k=1}^n a^r_{j,i_2,\ldots,i_k}x_{i_2}\cdots x_{i_k}\Big)_{j=1}^n,
\end{align*}
where in the last equality we used the symmetry of the coefficients $a^r_{i_1,\ldots,i_k}$.
Denoting $B_r = B_r^1\cup B_r^{-1}$ we see that on $B_r$, we have
\begin{align*}
|\nabla f(x)|_{\Psi_p} & \le Ck \sup_{y \in A_{\Psi,p}} \Big|\sum_{i_1,\ldots,i_k=1}^n a^r_{i_1,\ldots,i_k} y_{i_1} x_{i_2}\cdots x_{i_k}\Big|\\
& = Ck \sup_{y \in A_{\Psi,p}} |\langle A_r, y\otimes x^{\otimes(k-1)}\rangle|.
\end{align*}
Thus almost surely
\begin{displaymath}
|\nabla f(X)|_{\Psi_p} \le Ck \max_{1\le r \le N} \sup_{y \in A_{\Psi,p}}|\langle A^r, y\otimes X^{\otimes(k-1)}\rangle|
\end{displaymath}
The right-hand side above is a supremum of homogeneous forms of degree $k-1$ in $X$, moreover by separability of $A_{\Psi,p}$ it can be clearly approximated by suprema of a finite number of such forms. Therefore, by the induction assumption,
\begin{align}\label{eq:chaos_induction}
\Big\| |\nabla f(X)|_{\Psi_p} \Big\|_p \le& Ck \E \max_{1\le r \le N} \sup_{y \in A_{\Psi,p}}|\langle A^r, y\otimes X^{\otimes(k-1)}\rangle|
 \nonumber \\&+ C_{D,K,\alpha,\beta,k-1} \sum_{j=2}^k  \E \max_{1\le r\le N} \sup_{y\in A_{\Psi_p}}\sup_{y^2,\ldots,y^j \in A_{\Psi,p}} |\langle A^r,y\otimes y^2\otimes \cdots\otimes y^j \otimes X^{\otimes(k-j)}\rangle|\nonumber\\
=& C_{D,K,\alpha,\beta,k}  \sum_{j=1}^k  \E  \sup_{y^1,y^2,\ldots,y^j \in A_{\Psi,p}}\max_{1\le r\le N} |\langle A^r,y^1\otimes y^2\otimes \cdots\otimes y^j \otimes X^{\otimes(k-j)}\rangle|\nonumber
\end{align}
This ends the proof of \eqref{eq:Borell}, since by Theorem \ref{thm:Sobolev_centered},
\begin{displaymath}
\bn Z - \E Z \bn_p \le C_{D,K,\alpha,\beta} \Big\| |\nabla f(X)|_{\Psi_p} \Big\|_p.
\end{displaymath}
The second estimate of Theorem \ref{thm:Borell} follows now by the Chebyshev inequality.
\end{proof}

We will now pass to the proof of Theorem \ref{thm:higher_derivatives}. Let us start with the main tool, which is Corollary \ref{cor:comparison}.

\begin{proof}[Proof of Corollary \ref{cor:comparison}]
Relating the Legendre transform of $\tilde{\Phi}$ to the conjugation in a sense of Lemma~\ref{lemma:omega-omega-ast}, one can deduce that the growth condition~\eqref{ineq:growth-of-Phi} on $\tilde{\Phi}$ implies that $\tilde{\Phi}^\ast$ satisfies
\[
  K'^{-1} t^\alpha \le \frac{\tilde{\Phi}^\ast(tu)}{\tilde{\Phi}^\ast(u)} \le K' t^\beta
\]
for all $t \ge 1$ and $u > 0$, where $K' = C(K,\alpha,\beta)$. As a consequence, $\Psi$ satisfies~($G_{K',\alpha,\beta}$) and therefore, for $p \ge \beta$ the corollary is a direct consequence of Theorem~\ref{thm:Sobolev_centered} and Theorem \ref{thm:GK}.

For $2\le p < \beta$ we use the fact that if $\mu$ satisfies $mLSI(\Psi,D)$, then it also satisfies the Poincar\'e inequality
\begin{displaymath}
\Var(f) \le C_D \E |\nabla f|^2
\end{displaymath}
(see \cite{MR2351133}),
which as is well known (see e.g. \cite{MR2146071} or \cite{Milman_role_iso}) implies that
\begin{displaymath}
\|f - \E f\|_p \le C'_D p\bn|\nabla f|\bn_p.
\end{displaymath}
Moreover, due to the normalization $\Phi(1) = 1$ one can easily get $|x| \le C|x|_{\Psi_p}$ for $p\ge 1$,
which allows to deduce the corollary for $2\le p < \beta$.
\end{proof}

\begin{proof}[Proof of Theorem \ref{thm:higher_derivatives}] Given Corollary \ref{cor:comparison}, the proof follows with just formal changes the proof of Proposition 3.2. in \cite{Nonlipschitz}.
\end{proof}

\subsection{Proofs of results from Section \ref{sec:BCG}} \label{sec:BCG_proofs}
In this section we wil present the proof of Theorem \ref{thm:BCG}. Theorem \ref{thm:BCG-C2}  will then follow by specializing to $k = 2$.

Let us start with the following simple lemma.
\begin{lemma} \label{le:mth_derivative} In the setting of Theorem \ref{thm:BCG-C2}, for every $m < k$,
\begin{displaymath}
\bn |\D^m f - \E \D^m f|_2\bn_2 \le \sqrt{2}L\bn|\D^{m+1} f|_2\bn_2
\end{displaymath}
\end{lemma}

\begin{proof} We will regard $\D^m f$ as a vector in $(\R^n)^{\otimes m} \simeq \R^{n^m}$.
Let $X$ be a random vector distributed according to $\mu$ and $G$ a standard Gaussian vector in $\R^{n^m}$, independent of $X$. Then
\begin{align*}
\bn |\D^m f - \E \D^m f|_2\bn_2^2 &= \E_X |\D^m f(X) - \E_X \D^m f(X)|_2^2 = \E_G \E_X\langle  \D^m f(X) - \E_X \D^m f(X), G\rangle ^2\\
& \le 2L^2 \E_G \E_X |\D \langle  \D^m f(X),G \rangle|_2^2,
\end{align*}
where the second equality follows from the Fubini theorem and the inequality from the assumption \eqref{eq:BCG_assumption}, applied conditionally on $G$ to the function $x \mapsto \langle \D^m f(x),G\rangle$. Now, it is easy to see that
\begin{displaymath}
\E_G |\D \langle  \D^m f(X), G\rangle|_2^2 =  |\D^{m+1} f(X)|_2^2,
\end{displaymath}
which ends the proof.
\end{proof}

\begin{cor}\label{cor:mth_derivative}
    In the setting of Theorem \ref{thm:BCG-C2}, for all $k \ge 2$,
\begin{displaymath}
\bn |\nabla f |_2\bn_2 \le (\sqrt{2}L)^{k-1} \bn|\D^{k} f|_2\bn_2 + \sum_{m=1}^{k-1} (\sqrt{2}L)^{m-1} |\E_\mu \D^m f|_2.
\end{displaymath}
\end{cor}
\begin{proof}
An induction on $k$, using Lemma \ref{le:mth_derivative}.
\end{proof}

\begin{proof}[Proof of Theorem \ref{thm:BCG}]
By \eqref{eq:BCG_assumption}, we have
\begin{align}\label{eq:BCG_proof}
\|f - \E_\mu f\|_p \le L\sqrt{p}\bn|\nabla f|_2\bn_p.
\end{align}

It is easy to prove that $|\nabla|_2$ is locally Lipschitz and $\big|\nabla |\nabla f|_2\big|_2 \le |\D^2 f|_{\textup{op}}$ $\mu$-a.s. Indeed, we have for $x \in \R^n$ and $|h| \to 0$, by the triangle inequality and Taylor's expansion,
\begin{align*}
\frac{\Big| |\nabla f(x+h)|_2 - |\nabla f(x)|_2\Big|}{|h|} & \le \frac{|(\sum_{j=1}^n \frac{\partial f}{\partial x_j x_i }(x) h_j + o(|h|))_{i=1}^n|_2}{|h|} \\
& \le \frac{|\D^2 f(x) h|_2}{|h|} + o(1) \le |\D^2 f(x)|_{\textup{op}} + o(1),
\end{align*}
which via standard compactness arguments yields that $|\nabla f|_2$ is locally Lipschitz and that if $\nabla |\nabla f(x)|_2$ exists (which happens $\mu$-a.s.), then its Euclidean norm does not exceed $|\D^2 f(x)|_{\textup{op}}$.

Thus, another application of $\eqref{eq:BCG_assumption}$ gives
\begin{displaymath}
\bn|\nabla f|_2 \bn_p \le \E_\mu |\nabla f|_2+ \bn|\nabla f|_2- \E_\mu |\nabla f|_2 \bn_p\le \E_\mu|\nabla f|_2 + L\sqrt{p} \bn|\D^2 f|_{\textup{op}}\bn_p,
\end{displaymath}
which together with \eqref{eq:BCG_proof} gives the first inequality of \eqref{eq:BCG_statement}. The second inequality follows now by Corollary \ref{cor:mth_derivative}.
\end{proof}

\bibliographystyle{abbrv}
\bibliography{citations}
\end{document}